\documentclass[11pt, leqno, twoside]{article}
\usepackage{}

\usepackage{amsfonts}

\usepackage{amssymb}
\usepackage{amsmath}
\usepackage{amsthm}
\usepackage{xcolor}
\usepackage{mathrsfs}
\usepackage{txfonts}
\usepackage{extarrows}

\usepackage{indentfirst}

\usepackage{tikz}
\usepackage{subcaption}

\allowdisplaybreaks

\def\fz{\infty}
\def\az{\alpha}

\def\dz{\delta}
\def\Dz{{\Delta}}

\def\lz{\lambda}
\def\nn{{\mathbb{N}}}
\def\Oz{\Omega}
\def\rn{{\mathbb{R}^n}}

\def\lf{\left}
\def\r{\right}
\def\hs{\hspace{0.25cm}}
\def\ls{\lesssim}
\def\noz{\nonumber}

\def\vz{\varphi}
\def\wz{\widetilde}
\def\dis{\displaystyle}
\def\tzt{{\rm Aut}\,(G)}

\def\esup{\mathop\mathrm{\,ess\,sup\,}}

\newtheorem{theorem}{Theorem}[section]
\newtheorem{lemma}[theorem]{Lemma}

\newtheorem{proposition}[theorem]{Proposition}

\theoremstyle{definition}
\newtheorem{remark}[theorem]{Remark}
\newtheorem{definition}[theorem]{Definition}
\renewcommand{\appendix}{\par
   \setcounter{section}{0}%
   \setcounter{subsection}{0}%
   \setcounter{subsubsection}{0}%
   \gdef\thesection{\@Alph\c@section}%
   \gdef\thesubsection{\@Alph\c@section.\@arabic\c@subsection}%
   \gdef\theHsection{\@Alph\c@section.}%
   \gdef\theHsubsection{\@Alph\c@section.\@arabic\c@subsection}%
   \csname appendixmore\endcsname
 }

\numberwithin{equation}{section}

\begin{document}

\arraycolsep=1pt

\title{\bf\Large Time-Frequency Representations on Lorentz Spaces
over Locally Compact Abelian Groups\footnotetext {\hspace{-0.35cm}
2020 {\it Mathematics Subject Classification}. Primary 43A25;
Secondary 47G30, 42B10, 81S30, 46E30.
\endgraf {\it Key words and phrases}. Locally compact Abelian group, Lorentz space,
short-time Fourier transform, Wigner transform, Lieb's uncertainty principle.
\endgraf
This project is supported by the National
Natural Science Foundation of China (Grant Nos.~12371102 and 12371025).
}}
\author{Jun Liu, Yaqian Lu, Xianjie Yan
and Chi Zhang\footnote{Corresponding author, E-mail: zclqq32@cumt.edu.cn}}
\date{}
\maketitle

\vspace{-0.9cm}

\begin{center}
\begin{minipage}{13cm}
{\small {\bf Abstract}\quad
Let $G$ be a locally compact Abelian group with a fixed Haar measure and, denote by $\widehat{G}$ its dual group. In this article, the authors obtain various boundedness of the short-time Fourier transform on Lorentz
spaces:
$$L^{p_1,u}(G)\times L^{p_2,v}(G)\to L^{q,w}(G\times\widehat{G})$$
with the indexes satisfying appropriate relations. These results are then used to prove the corresponding
boundedness of $\tau$-Wigner transforms and $\tau$-Weyl operators. As an application, the Lieb's uncertainty principle in the context of Lorentz spaces is finally investigated.
All these results are new even for the case when $G$ is finite.}
\end{minipage}
\end{center}

\section{Introduction}\label{s1}

Time-frequency analysis \cite{cg03,Gr01}, which originated from the need to formulate mathematical models
for acoustic signals, has evolved into a fundamental framework in harmonic analysis and signal theory.
It involves analyzing signals not merely in isolation or via their Fourier transforms, but through a rich
set of representations constructed in the time-frequency space. Nowadays, as a mature and substantial branch
of mathematics, time-frequency analysis has found important applications in diverse fields such as
quantization, acoustics, geophysics and biomedicine. At a more theoretical level, its constructs have revealed connections to group representations, operator theory and $C^*$-algebras.

It is well known that the classical Fourier analysis, while capable of precisely characterizing the global spectral structure of a signal, suffers from a fundamental limitation: it can reveal the overall intensity of various frequency components present in a signal but fails to specify their temporal location. To overcome
this constraint, various other types of representations of signals have been developed. The core idea is that
if a global transformation obscures temporal information, the analysis should instead focus on local segments
of the signal. This is achieved by using a window function, which is localized in time, such as a Gaussian
function concentrated near the origin. Among these representations, the short-time Fourier transform (STFT)
and the Wigner transform are maybe the two most typical representatives. Here, in order to facilitate the
reader to know more development related to this topic, we give some references;
for instance, \cite{bdo07,bdo09,bdo10,bos25,ggrt,gjm,Gr21,Ni23,nt22,nt20}.

On another hand,
the study of Lorentz spaces can be traced back to the early 1950's, following the groundbreaking work of
Lorentz \cite{l50}. These spaces generalize the classical Lebesgue spaces $L^p$ and are known to be
intermediate spaces of Lebesgue spaces via the real interpolation; see, for instance, \cite{bs88,bl76,p63}.
We should point out that, due to their own fine structures, Lorentz spaces appear frequently in the study on
various critical or endpoint analysis problems from many different research fields and there exist enormous literatures on this subject, here we only mention several recent articles, respectively, from harmonic analysis
(see, for instance, \cite{bos25,fx23,gra3,jcwyy,jwyyz}), partial differential equations
(see, for instance, \cite{bc23,llsy,zyy})
and stochastic processes (see, for instance, \cite{jwxy,Os21,wxy}).

In particular, one recent significant application of Lorentz spaces in time-frequency analysis
was due to Boggiatto et al. \cite{bos25}, in which the authors established boundedness properties of
the short-time Fourier
transform on Lorentz spaces in the setting of classical Euclidean spaces. Based on this, they further
obtained the corresponding conclusions for the (cross-)Wigner and $\tau$-Wigner transforms.
An application to the uncertainty principle in the context of Lorentz spaces was also presented.
Moreover, at the end of the introduction of \cite{bos25}, the authors stated that
``Thirdly, in the lines of some recent papers devoted to find optimizers of uncertainty inequalities on
Euclidean spaces, Riemannian manifolds, and locally compact Abelian groups (see [14, 15, 17]),
a recasting of the whole subject in some of these more general settings could be an
interesting topic for a further research." Motivated by this, in the present article, we extend
the main results of \cite{bos25} from Euclidean spaces to locally compact Abelian groups.

Let $G$ be a locally compact Abelian group with a fixed Haar measure.
The remainder of this article is organized as follows.

In Section \ref{s2},
we first recall the definition of Lorentz spaces on $G$, denoted by $L^{p,q}(G)$, and then state our main
results including several boundedness of the short-time Fourier transform on $L^{p,q}(G)$ (see Theorems
\ref{t1'}, \ref{t1} and \ref{t2}) as well as the boundedness of $\tau$-Wigner transforms and $\tau$-Weyl
operators on $L^{p,q}(G)$ (see Theorems \ref{t3} and \ref{t4}). These boundedness are finally used to
explore the uncertainty principle related to $L^{p,q}(G)$; see Theorem \ref{t5}.

Via establishing a Hardy--Littlewood--Stein inequality on Lorentz spaces $L^{p,q}(G)$, we first show
Theorem \ref{t1'} in Section \ref{s3}. Then the boundedness of the short-time Fourier transform on
Lebesgue spaces $L^{p}(G)$ is given in Lemma \ref{3l1}. This result allows us to
conclude that the short-time Fourier transform is of restricted weak type
$(2,2;2)$, $(1,\fz;\fz)$ and $(\fz,1;\fz)$. Applying the technique of restricted weak type operators,
we further derive a key estimate for Calder\'{o}n operators on Lorentz spaces (see Lemma \ref{3l3}),
which is crucial for proving Theorem \ref{t1}.
The proof of Theorem \ref{t2} follows an argument similar to that of Theorem \ref{t1}.
Leveraging the known boundedness of tensor product operators on $L^{p,q}(G)$
and two auxiliary Lorentz quasi-norms inequalities established in Lemmas \ref{4l2} and \ref{4l3},
we prove Theorem \ref{t3} in Section \ref{s4}.
Finally, Section \ref{s5} is devoted to the proof of Theorem \ref{t5}.
We note that the core
techniques used in this article to show our main results are adapted from \cite{bos25,Sa14}.

Hereinafter, let $\nn:=\{1,2,\ldots\}$, $\mathbb{R}_+:=(0,\fz)$ and,
for any set $E$, denote by $\mathbf{1}_E$ its \emph{characteristic function}.
We use $C$ to denote a positive constant
which is independent of the main parameters involved, whose value may vary from line to line.
The \emph{symbol} $g\ls h$ means $g\le Ch$ and,
if $g\ls h\ls g$, then we write $g\sim h$. If $g\le Ch$ and $h=f$ or $h\le f$,
we then write $g\ls h=f$ or $g\ls h\le f$, \emph{rather than}$g\ls h\sim f$
or $g\ls h\ls f$. In addition, for any $q\in[1,\fz]$, we use $q'$ to denote
its \emph{conjugate index}, that is, $1/q+1/q'=1$.

\section{Preliminaries and main results}\label{s2}

In this section, we first recall some notions
which will be used throughout this article and then state the main results of this article.

\subsection{Lorentz spaces on LCA groups $L^{p,q}(G)$}\label{s2.1}

We commence with the definition of Lorentz spaces from \cite{bs88,gra}. Let $G$ be a locally compact
Abelian (in short, LCA) group with a Haar measure $\mu$. For any measurable function $f$ on $G$,
the distribution function of $f$ is the function $d_f$ defined on $[0,\fz)$ as follows:
\begin{align}\label{2e0}
d_f(\alpha):=\mu(\{x\in G:\ |f(x)|>\alpha\})
\end{align}
and, the non-increasing rearrangement of $f$, denoted by $f^\ast$, is defined by setting
$$f^\ast(t):=\inf\lf\{\alpha\in(0,\fz):\ d_f(\alpha)\le t\r\}, \quad t\in[0,\fz).$$
Moreover, for any $p,\,q\in(0,\fz]$, the \emph{Lorentz space} $L^{p,q}(G)$
is defined to be the space of all $\mu$-measurable functions $f$ on $G$ with finite $L^{p,q}(G)$
quasi-norm $\|f\|_{L^{p,q}(G)}$ given by
\begin{align}\label{2e1}
\|f\|_{L^{p,q}(G)}:=\left\{
\begin{array}{cl}
&\lf\{\dis\int_0^\fz\lf[t^{\frac{1}{p}}{f^\ast(t)}\r]
^q\frac{dt}{t}\r\}^{\frac{1}{q}}
\hspace{0.5cm} {\rm if}\ q\in(0,\fz),\\
&\dis\sup_{t\in(0,\fz)}\lf[t^
\frac{1}{p}f^\ast(t)\r]
\hspace{1.58cm}{\rm if}\ q=\infty.
\end{array}\r.
\end{align}

\begin{remark}\label{2r1}
\begin{enumerate}
\item[{\rm(i)}] For any $p,\,q\in(0,\fz]$, the Lorentz space $L^{p,q}(G)$ is are complete with respect
to their quasi-norm; see \cite[Theorem 1.4.11]{gra}.
\item[{\rm(ii)}] We remark that $L^{p,p}(G)=L^p(G)$ for each $p\in(0,\infty]$, where $L^p(G)$
denotes the Lebesgue space defined to be the set of all $\mu$-measurable functions $f$ on
$G$ such that, when $p\in(0,\fz)$,
$$\|f\|_{L^p(G)}:=\lf[\int_G|f(x)|^p\,d\mu(x)\r]^{1/p}<\fz$$
and
$$\|f\|_{L^\fz(G)}:=\esup_{x\in G}|f(x)|<\fz.$$
\item[{\rm(iii)}] We point out that, \eqref{2e1} yields in general only quasi-norms, however
we can replace in \eqref{2e1} the
function $f^*$ by $f^{**}$, getting a norm for the Lorentz space $L^{p,q}(G)$ equivalent to
the quasi-norm in \eqref{2e1}, where $p\in(1,\fz]$, $q\in[1,\fz]$ and $f^{**}$ denotes the
maximal function of $f^*$ defined by setting, for any $t\in(0,\fz)$,
\begin{align}\label{2e2'}
f^{**}(t):=\frac1t\int_0^t f^*(u)\,du.
\end{align}
\end{enumerate}
\end{remark}

\subsection{Boundedness of the STFT on $L^{p,q}(G)$}\label{s2.2}

The aim of this subsection is to display the result on the boundedness of the short-time Fourier transform
(in short, STFT) on $L^{p,q}(G)$. To this end, we need the notion of the STFT in the setting of LCA groups;
see \cite{Ni24,Gr98,Gr01} for more details.

Let $G$ be a LCA group and, denote by $\widehat{G}$ its dual group, which is also a LCA group with the
topology of the uniform convergence on the compact subsets (\cite[Theorems 23.13 and 23.15]{hr63}).
The groups $G$ and $\widehat{G}$ are equipped, respectively, with Haar measures $\mu$ and $\mu^*$
such that the Plancherel formula holds true. The Haar measure on $G\times\widehat{G}$ is given by the
corresponding Radon product measure. We denote by $\langle x,\xi\rangle$ the value of $\xi\in\widehat{G}$
at $x\in G$. The pairing $G\times\widehat{G}\to\mathbb{T}$ (the multiplicative group of complex numbers of
modulus 1) given by $(x,\xi)\to \langle x,\xi\rangle$ is therefore well defined and $|\langle x,\xi\rangle|=1$.
We will write the group laws in $G$ and $\widehat{G}$ additively, hence
$$\langle x+y,\xi\rangle=\langle x,\xi\rangle\langle y,\xi\rangle\ {\rm and}\
\langle x,\xi+\zeta\rangle=\langle x,\xi\rangle\langle x,\zeta\rangle,\quad
\forall\,x,\,y\in G,\ \forall\,\xi,\,\zeta\in\widehat{G}.$$

For any $f\in L^1(G)$, the \emph{Fourier transform} $\widehat{f}$ (or denoted by $f^\wedge$)
is defined by setting, for any $\xi\in\widehat{G}$,
$$\widehat{f}(\xi):=\int_Gf(x)\overline{\langle x,\xi\rangle}\,d\mu(x).$$
For any $x\in G$ and $\xi\in\widehat{G}$, the \emph{translation} and \emph{modulation operators} $T_x$ and
$M_\xi$ on $L^2(G)$, and the corresponding \emph{time-frequency shifts} $\pi(x,\xi)$ are defined, respectively,
as
$$T_xf(y):=f(y-x),\quad M_\xi f(y):=\langle y,\xi\rangle f(y)$$
and
$$\pi(x,\xi)f(y):=M_\xi T_x f(y),\quad \forall\,y\in G.$$
These three operators are unitary on $L^2(G)$. The inner product in $L^2(G)$ is denoted by
$\langle\cdot,\cdot\rangle_{L^2(G)}$. For any $f,\,g\in L^2(G)$, \emph{the short-time Fourier transform}
(STFT) of $f$ with respect to $g$ (called the \emph{window function}) is the complex-valued function on
$G\times\widehat{G}$ given by
\begin{align}\label{2e1'}
V_gf(x,\xi)&:=\lf\langle f,\pi(x,\xi)g\r\rangle_{L^2(G)}\\
&=\int_Gf(y)\overline{g(y-x)}\,\overline{\langle y,\xi\rangle}\,d\mu(y),\quad
\forall\,(x,\xi)\in G\times\widehat{G}.\noz
\end{align}

\begin{remark}\label{2r2}
By \cite[Proposition 2.1]{Ni24}, we know that, for any $f,\,g\in L^2(G)$, the STFT $V_gf$ is a continuous function on $G\times\widehat{G}$, which vanishes at infinity. Moreover, the set
$\{(x,\xi)\in G\times\widehat{G}:\ V_gf(x,\xi)\neq0\}$ is $\sigma$-compact.
\end{remark}

The main results of this subsection are stated as follows.

\begin{theorem}\label{t1'}
Let $q\in(2,\fz)$, $p_1,\,p_2\in(1,\fz)$, and $u,\,v,\,w\in[1,\fz]$ such that
$$\frac1{p_1}+\frac1{p_2}=1-\frac1q\quad{ and}\quad\frac1u+\frac1v\ge\frac1w.$$
Then there exists a positive constant $C$ such that, for any $(f,g)\in L^{p_1,u}(G)\times L^{p_2,v}(G)$,
\begin{align*}
\lf\|V_gf\r\|_{L^{q,w}(G\times\widehat{G})}\le C\|f\|_{L^{p_1,u}(G)}\|g\|_{L^{p_2,v}(G)}.
\end{align*}
\end{theorem}

\begin{theorem}\label{t1}
Let $q\in(2,\fz)$, $p\in[q',q]$ with $p\neq2$, and $u,\,v,\,w\in[1,\fz)$ such that
$$\frac1u+\frac1v\ge1+\frac1w.$$
Then there exists a positive constant $C$ such that, for any $(f,g)\in L^{p',u}(G)\times L^{p,v}(G)$,
\begin{align*}
\lf\|V_gf\r\|_{L^{q,w}(G\times\widehat{G})}\le C\|f\|_{L^{p',u}(G)}\|g\|_{L^{p,v}(G)}.
\end{align*}
\end{theorem}

In Theorem \ref{t1}, we require $p\neq2$; for the case when $p=2$, we have the following result.

\begin{theorem}\label{t2}
Let $q\in(2,\fz)$.
Then there exists a positive constant $C$ such that, for any $(f,g)\in L^{2,1}(G)\times L^{2,1}(G)$,
\begin{align}\label{2e3}
\lf\|V_gf\r\|_{L^{q,1}(G\times\widehat{G})}\le C\|f\|_{L^{2,1}(G)}\|g\|_{L^{2,1}(G)}.
\end{align}
\end{theorem}

\begin{remark}\label{2r3}
\begin{enumerate}
\item[{\rm(i)}] If we take $(G,\mu)$ as the $n$-dimensional Euclidean space $\rn$ with the Lebesgue measure,
then $\widehat{G}=\rn$. In this case, Theorems \ref{t1'}, \ref{t1} and \ref{t2} above are just
\cite[Proposition 3]{Sa14} and
\cite[Theorems 3.6 and 3.7]{bos25}, respectively. Moreover, \cite[Theorem 3.8]{bos25} shows that
in many of the cases when the hypotheses of Theorems 3.6 and 3.7 therein do not hold, the STFT fails to
be bounded. To be precise, \emph{if $\max\{p,p'\}>q$ and $u,v,w\in[1,\fz]$, then the STFT is not bounded
as an operator from $L^{p',u}(\rn)\times L^{p,v}(\rn)$ to $L^{q,w}(\rn\times\rn)$}; however, we can not
obtain the analogous result in the present setting of LCA groups. This is related to the fact that LCA
groups are more general underlying spaces than the Euclidean ones, and this causes much more difficulties
when trying to prove the situations where boundedness does not holds.
Although we think \cite[Theorem 3.8]{bos25} in the setting of LCA groups is also true,
how to verify this is still unclear so far.
\item[{\rm(ii)}]
As pointed out in \cite{bos25}, the conclusions obtained in Theorems \ref{t1'}, \ref{t1} and \ref{t2}
do not cover all the cases of Lebesgue spaces. For instance, it is well known that the STFT is bounded as
an operator from $L^{2}(G)\times L^{2}(G)$ to $L^{2}(G)$ (see \cite{Ni24}); this needs to take
$p=q=u=v=w=2$ in Theorem \ref{t1} which is not allowed.
Obviously, it can not be derived from Theorem \ref{t2} either.
\end{enumerate}
\end{remark}

\subsection{Boundedness of $\tau$-Wigner transforms and $\tau$-Weyl operators on $L^{p,q}(G)$}\label{s2.3}

We first extend the definitions of both $\tau$-Wigner transforms and $\tau$-Weyl operators
from the $n$-dimensional Euclidean space $\rn$
to LCA groups. Let $G$ be a LCA group with a fixed Haar measure $\mu$. For any given continuous endomorphism
$\tau$ on $G$ and any $f,\,g\in L^2(G)$, the \emph{$\tau$-Wigner transform} is defined as
\begin{align}\label{2e4}
W_\tau(f,g)(x,\xi):=\int_Gf(x+\tau(y))\overline{g(x-(I-\tau)y)}\,\overline{\langle y,\xi\rangle}\,d\mu(y),
\quad \forall\,(x,\xi)\in G\times\widehat{G},
\end{align}
here and thereafter, $I$ denotes the identity mapping on $G$.
Moreover, for any given $\vz\in L^2(G\times\widehat{G})$, the \emph{$\tau$-Weyl operator}
(or called \emph{$\tau$-Weyl quantization}) $\mathcal{W}_\tau^\vz$ is defined by setting, for any
$f,\,g\in L^2(G)$,
\begin{align*}
\lf\langle \mathcal{W}_\tau^\vz f,\,g\r\rangle_{L^2(G)}
:=\lf\langle \vz,\,W_\tau(f,g)\r\rangle_{L^2(G)}.
\end{align*}

\begin{remark}\label{2r4}
\begin{enumerate}
\item[{\rm(i)}] Denote by $\tzt$ the set of all topological automorphisms on $G$. If $\tau\in\tzt$,
then there exists a constant $\Dz_\tau\in(0,\fz)$ such that, for any $\mu$-measurable function $f$,
$$\int_Gf(\tau(x))\,d\mu(x)=\Dz_\tau^{-1}\int_Gf(x)\,d\mu(x)
=\Dz_{\tau^{-1}}\int_Gf(x)\,d\mu(x),$$
where $\tau^{-1}$ denotes the inverse of $\tau$ and $\Dz_\tau$ is called the modulus of $\tau$.
\item[{\rm(ii)}]
If $\tau(y)\equiv\theta$ for $\mu$-almost every $y\in G$, where $\theta$ is the identity element of $G$,
then \eqref{2e4} becomes
\begin{align*}
W_\theta(f,g)(x,\xi)&=\int_Gf(x)\overline{g(x-y)}\,\overline{\langle y,\xi\rangle}\,d\mu(y)\\
&=f(x)\int_G\overline{g(z)}\,\overline{\langle x-z,\xi\rangle}\,d\mu(z)\\
&=f(x)\overline{\langle x,\xi\rangle}\int_G\overline{g(z)}\,\langle z,\xi\rangle\,d\mu(z)\\
&=f(x)\overline{\langle x,\xi\rangle}\,\overline{\widehat{g}(\xi)},
\qquad \forall\,(x,\xi)\in G\times\widehat{G},
\end{align*}
which is also called the Rihaczek transform.
\item[{\rm(iii)}]
If $\tau(y)\equiv y$ for $\mu$-almost every $y\in G$, that is, $\tau=I$ in the sense of $\mu$-a.e.,
then \eqref{2e4} turns into
\begin{align*}
W_I(f,g)(x,\xi)&=\int_Gf(x+y)\overline{g(x)}\,\overline{\langle y,\xi\rangle}\,d\mu(y)\\
&=\overline{g(x)}\langle x,\xi\rangle\,\widehat{f}(\xi),
\qquad \forall\,(x,\xi)\in G\times\widehat{G},
\end{align*}
which is also called the conjugate Rihaczek transform.
\end{enumerate}
\end{remark}

Using Theorem \ref{t1}, we obtain the following boundedness of $\tau$-Wigner transforms and $\tau$-Weyl
operators on Lorentz spaces.

\begin{theorem}\label{t3}
\begin{enumerate}
\item[{\rm(i)}]
Let $q\in(2,\fz)$, $p_1,\,p_2\in(1,\fz)$, and $u,\,v,\,w\in[1,\fz]$ such that
\begin{align}\label{2e5'}
\frac1{p_1}+\frac1{p_2}=1-\frac1q\quad{ and}\quad\frac1u+\frac1v\ge\frac1w.
\end{align}
If $\tau\in\tzt$ with its modulus $\Dz_\tau\in(0,1)$, then there exists a positive constant $C_1$
such that, for any $(f,g)\in L^{p_1,u}(G)\times L^{p_2,v}(G)$,
\begin{align*}
\lf\|W_\tau(f,g)\r\|_{L^{q,w}(G\times\widehat{G})}\le C_1\|f\|_{L^{p_1,u}(G)}\|g\|_{L^{p_2,v}(G)}.
\end{align*}

\item[{\rm(ii)}]
Let $q\in(2,\fz)$, $p\in[q',q]$ with $p\neq2$, and $u,\,v,\,w\in[1,\fz)$ such that
\begin{align}\label{2e5}
\frac1u+\frac1v\ge1+\frac1w.
\end{align}
If $\tau\in\tzt$ with its modulus $\Dz_\tau\in(0,1)$,
then there exists a positive constant $C_2$ such that, for any $(f,g)\in L^{p',u}(G)\times L^{p,v}(G)$,
\begin{align*}
\lf\|W_\tau(f,g)\r\|_{L^{q,w}(G\times\widehat{G})}\le C_2\|f\|_{L^{p',u}(G)}\|g\|_{L^{p,v}(G)}.
\end{align*}
\item[{\rm(iii)}]
Let $p\in(2,\fz)$ and $u,\,v,\,w\in[1,\fz]$ satisfying \eqref{2e5}.
Then there exists a positive constant $C_3$ such that, for any $(f,g)\in L^{p,u}(G)\times L^{p',v}(G)$,
\begin{align*}
\lf\|W_\theta(f,g)\r\|_{L^{p,w}(G\times\widehat{G})}\le C_3\|f\|_{L^{p,u}(G)}\|g\|_{L^{p',v}(G)},
\end{align*}
where $W_\theta(f,g)$ is the Rihaczek transform as in Remark \ref{2r4}(ii).
\item[{\rm(iv)}]
Let $p\in(2,\fz)$ and $u,\,v,\,w\in[1,\fz]$ satisfying \eqref{2e5}.
Then there exists a positive constant $C_4$ such that, for any $(f,g)\in L^{p',u}(G)\times L^{p,v}(G)$,
\begin{align*}
\lf\|W_I(f,g)\r\|_{L^{p,w}(G\times\widehat{G})}\le C_4\|f\|_{L^{p',u}(G)}\|g\|_{L^{p,v}(G)},
\end{align*}
where $W_I(f,g)$ is the conjugate Rihaczek transform as in Remark \ref{2r4}(iii).
\end{enumerate}
\end{theorem}

\begin{definition}
A measurable subset $A$ of a measure space $(X,\lz)$ is called an
atom if $\lz(A)>0$ and every measurable subset $B$ of $A$ has measure either equal to
zero or equal to $\lz(A)$. A measure space $(X,\lz)$ is called \emph{nonatomic} if it contains no
atoms. In other words, $X$ is nonatomic if and only if for any $A\subset X$ with $\lz(A)>0$,
there exists a proper subset $B\subset A$ with $\lz(B)>0$ and $\lz(A\backslash B)>0$.
\end{definition}

Theorem \ref{t3}, \cite[Proposition 2.4]{bdo10} and \cite[Theorem 1.4.16]{gra}
immediately implies the boundedness of the $\tau$-Weyl
quantization on Lorentz spaces; the details are omitted.

\begin{theorem}\label{t4}
Assume that $(G,\mu)$ is nonatomic.
\begin{enumerate}
\item[{\rm(i)}]
Let $w\in(1,\fz)$, $q\in(2,\fz)$, $p_1,\,p_2\in(1,\fz)$, $u\in(1,\fz)$ and $v\in[1,\fz]$ satisfying
\eqref{2e5'}. If $\tau\in\tzt$ with its modulus $\Dz_\tau\in(0,1)$, then the quantization
\begin{align*}
\vz\in L^{q',w'}\lf(G\times\widehat{G}\r)\rightarrow
\mathcal{W}_\tau^\vz\in B\lf(L^{p_2,v}(G),L^{p_1',u'}(G)\r)
\end{align*}
is continuous.
\item[{\rm(ii)}]
Let $w\in(1,\fz)$, $q\in(2,\fz)$, $p\in[q',q]$ with $p\neq2$, $u\in(1,\fz)$ and $v\in[1,\fz)$ satisfying
\eqref{2e5}. If $\tau\in\tzt$ with its modulus $\Dz_\tau\in(0,1)$,
then the quantization
\begin{align*}
\vz\in L^{q',w'}\lf(G\times\widehat{G}\r)\rightarrow \mathcal{W}_\tau^\vz\in B\lf(L^{p,v}(G),L^{p,u'}(G)\r)
\end{align*}
is continuous.
\item[{\rm(iii)}]
Let $p\in(2,\fz)$, $u,\,w\in(1,\fz)$ and $v\in[1,\fz]$ satisfying \eqref{2e5}.
Then the quantization
\begin{align*}
\vz\in L^{p',w'}\lf(G\times\widehat{G}\r)\rightarrow \mathcal{W}_\theta^\vz\in B\lf(L^{p',v}(G),L^{p',u'}(G)\r)
\end{align*}
is continuous.
\item[{\rm(iv)}]
Let $p\in(2,\fz)$, $u,\,w\in(1,\fz)$ and $v\in[1,\fz]$ satisfying \eqref{2e5}.
Then the quantization
\begin{align*}
\vz\in L^{p',w'}\lf(G\times\widehat{G}\r)\rightarrow \mathcal{W}_I^\vz\in B\lf(L^{p,v}(G),L^{p,u'}(G)\r)
\end{align*}
is continuous.
\end{enumerate}
\end{theorem}

\begin{remark}\label{2r5}
Particularly,
if $(G,\mu)$ is the $n$-dimensional Euclidean space $\rn$ with the Lebesgue measure,
then $\widehat{G}=\rn$. Now, we may take the continuous endomorphism $\tau$ in \eqref{2e4} as
$$\tau:\ x\mapsto \epsilon x\ {\rm with\ any\ given}\ \epsilon\in(0,1].$$
Then the $\tau$-Wigner transform defined in \eqref{2e4} goes back to
\begin{align*}
W_\epsilon(f,g)(x,\xi):=\int_\rn f(x+\epsilon y)\overline{g(x-(1-\epsilon)y)}\,
e^{-2\pi iy\cdot\xi}\,dy,
\quad \forall\,(x,\xi)\in \rn\times\rn,
\end{align*}
where $y\cdot\xi=\sum_{k=1}^n y_k\xi_k$ is the usual inner product on $\rn$. In this sense,
Theorems \ref{t3} and \ref{t4} above include
\cite[Theorems 4.1 and 4.2]{bos25} as special cases, respectively.
\end{remark}

\subsection{An application to the uncertainty principle}\label{s2.4}

As an application of Theorems \ref{t1'} and \ref{t1}, we have the following result.

\begin{theorem}\label{t5}
\begin{enumerate}
\item[{\rm(i)}]
Let $q\in(2,\fz)$, $p_1,\,p_2\in(1,\fz)$ and $u,\,v\in[1,\fz]$ such that
\begin{align*}
\frac1{p_1}+\frac1{p_2}=1-\frac1q\quad{ and}\quad0\le\frac1u+\frac1v\le1.
\end{align*}
Then, for any $(f,g)\in L^{p_1,u}(G)\times L^{p_2,v}(G)$ and $\Omega\in G\times\widehat{G}$ satisfying
\begin{align}\label{2e7}
\int_{\Omega}\lf|V_g(f)(x,\xi)\r|^2\,d\mu(x)\mu^*(\xi)\ge\varepsilon
\end{align}
with $\varepsilon\in(0,\fz)$, we have
\begin{align}\label{2e7'}
\lf(\mu\times\mu^*\r)(\Omega)\ge C\varepsilon^{s/2}\|f\|_{L^{p_1,u}(G)}^{-1}\|g\|_{L^{p_2,v}(G)}^{-1},
\end{align}
where $s\in(1,\fz)$ such that $\frac1q+\frac1s=\frac12$,
and $C$ is a positive constant independent of $f,\,g$ and $\Omega$.
\item[{\rm(ii)}]
Let $q\in(2,\fz)$, $p\in[q',q]$ with $p\neq2$, and $u,\,v\in[1,\fz)$ such that
\begin{align}\label{2e6}
\frac1u+\frac1v>1.
\end{align}
Then, for any $(f,g)\in L^{p',u}(G)\times L^{p,v}(G)$ and $\Omega\in G\times\widehat{G}$ satisfying
\eqref{2e7} with $\varepsilon\in(0,\fz)$, we also have
\begin{align}\label{2e8}
\lf(\mu\times\mu^*\r)(\Omega)\ge \wz C\varepsilon^{s/2}\|f\|_{L^{p',u}(G)}^{-1}\|g\|_{L^{p,v}(G)}^{-1},
\end{align}
where $s\in(1,\fz)$ such that $\frac1q+\frac1s=\frac12$,
and $\wz C$ is a positive constant independent of $f,\,g$ and $\Omega$.
\end{enumerate}
\end{theorem}

\begin{remark}\label{2r6}
For simplicity, besides the same conditions as in Theorem \ref{t5}, assume that
$$\|f\|_{L^{p_1,u}(G)}=1=\|g\|_{L^{p_2,v}(G)}\quad {\rm in\ (i)}$$
and
$$\|f\|_{L^{p',u}(G)}=1=\|g\|_{L^{p,v}(G)}\quad {\rm in\ (ii)}.$$
Then both \eqref{2e7'} and \eqref{2e8} become
\begin{align*}
\lf(\mu\times\mu^*\r)(\Omega)\gtrsim\varepsilon^{s/2}=\varepsilon^{\frac q{q-2}}
\end{align*}
for any $q\in(2,\fz)$.
Recall that the Lieb's uncertainty principle on a LCA group $G$ was proved by K. Gr\"{o}chenig \cite{Gr98};
see also \cite{Ni24}, which gives us
\begin{align*}
\lf(\mu\times\mu^*\r)(\Omega)\geq\lf(\frac p2\r)^{\frac{2n}{p-2}}\varepsilon^{\frac p{p-2}}
\sim \varepsilon^{\frac p{p-2}}
\end{align*}
for any $p\in(2,\fz)$, provided $f,\,g\in L^2(G)$ with $\|f\|_{L^{2}(G)}=1=\|g\|_{L^{2}(G)}$, and
$$\int_{\Omega}\lf|V_g(f)(x,\xi)\r|^2\,d\mu(x)\mu^*(\xi)\ge\varepsilon,$$
where $n\in\nn$ such that $G$ is topologically isomorphic to $\rn\times G_0$, and $G_0$ is a LCA group
containing a compact open subgroup.

In this sense, Theorem \ref{t5} extends the Lieb's uncertainty principle (the size control of $\Oz$)
from the setting of Lebesgue spaces to Lorentz spaces.
\end{remark}

\section{Proofs of Theorems \ref{t1'}, \ref{t1} and \ref{t2}}\label{s3}

To prove Theorem \ref{t1'}, we need the following Hardy--Littlewood--Stein inequality
and H\"{o}lder inequality on Lorentz spaces (\cite[Theorem 3.4]{ON63}).

\begin{lemma}\label{4l1}
Let $p\in(1,2)$ and $q\in(0,\fz]$. Then there exists a positive constant $C$ such that
$$\lf\|\widehat{f}\,\r\|_{L^{p',q}(G)}\le C\lf\|f\r\|_{L^{p,q}(G)}.$$
\end{lemma}

\begin{proof}
Notice that, for any $f\in L^{2}(G)$,
\begin{align*}
\lf\|\widehat{f}\,\r\|_{L^{2}(\widehat{G})}=\|f\|_{L^{2}(G)}
\end{align*}
and, for any $f\in L^{1}(G)$,
\begin{align*}
\lf\|\widehat{f}\,\r\|_{L^{\fz}(\widehat{G})}\ls\|f\|_{L^{1}(G)}.
\end{align*}
By the general Marcinkiewicz interpolation theorem (see \cite[Theorem 5.3.2]{bl76})
with the indexes therein taken as follows:
$$p_0=2=r_0,\ q_0=2=s_0,\ p_1=1=r_1,\ q_1=\fz=s_1\quad{\rm and}\quad r=q,$$
it is easy to check Lemma \ref{4l1} is true.
\end{proof}

\begin{proposition}\label{5p1}
Let $p,\,p_1,\,p_2\in(1,\fz)$ and $q,\,q_1,\,q_2\in(0,\fz]$ satisfying
$$\frac1{p_1}+\frac1{p_2}=\frac1p\quad{\rm and}\quad
\frac1{q_1}+\frac1{q_2}\ge\frac1q.$$
Then, for any $f\in L^{p_1,q_1}(G)$ and $g\in L^{p_2,q_2}(G)$,
\begin{align*}
\|fg\|_{L^{p,q}(G)}\le p'\|f\|_{L^{p_1,q_1}(G)}\|g\|_{L^{p_2,q_2}(G)}.
\end{align*}
\end{proposition}

Now, we give the proof of Theorem \ref{t1'}.

\begin{proof}[\textbf{Proof of Theorem \ref{t1'}}]
Observe that, for any given $x\in G$,
\begin{align}\label{3e1'}
V_gf(x,\xi)=\lf[f\overline{T_xg}\r]^\wedge(\xi),\quad\forall\,\xi\in\widehat{G}.
\end{align}
By this, Lemma \ref{4l1} and Proposition \ref{5p1}, we conclude that
\begin{align*}
\lf\|V_gf\r\|_{L^{q,w}(G\times\widehat{G})}
&=\lf\|\lf[f\overline{T_xg}\r]^\wedge\r\|_{L^{q,w}(G\times\widehat{G})}
\ls\lf\|f\overline{T_xg}\r\|_{L^{q',w}(G\times\widehat{G})}\\
&\ls q\|f\|_{L^{p_1,u}(G)}\|T_xg\|_{L^{p_2,v}(G)}
\sim\|f\|_{L^{p_1,u}(G)}\|g\|_{L^{p_2,v}(G)}.
\end{align*}
This finishes the proof of Theorem \ref{t1'}.
\end{proof}

To show Theorems \ref{t1} and \ref{t2}, we need some technical lemmas.
First, the boundedness of the STFT on Lebesgue spaces is required, which is also independent interest.

\begin{lemma}\label{3l1}
Let $q\in[2,\fz]$ and $p\in[q',q]$. Then, for any $(f,g)\in L^{p'}(G)\times L^{p}(G)$,
\begin{align*}
\lf\|V_gf\r\|_{L^{q}(G\times\widehat{G})}\le \|f\|_{L^{p'}(G)}\|g\|_{L^{p}(G)}.
\end{align*}
\end{lemma}

\begin{proof}
Applying \eqref{3e1'} and the Plancherel theorem for the Fourier transform, we find that
\begin{align*}
\lf\|V_gf\r\|_{L^{2}(G\times\widehat{G})}
&=\lf\|\lf[f\overline{T_xg}\r]^\wedge\r\|_{L^{2}(G\times\widehat{G})}\\
&=\lf\|f\overline{T_xg}\r\|_{L^{2}(G\times G)}
=\|f\|_{L^{2}(G)}\|g\|_{L^{2}(G)}.
\end{align*}
In addition, for any $r\in[1,\fz]$, it follows from \eqref{2e1'} and the H\"{o}lder inequality that
\begin{align*}
\lf\|V_gf\r\|_{L^{\fz}(G\times\widehat{G})}\le\|f\|_{L^{r}(G)}\|g\|_{L^{r'}(G)}.
\end{align*}
Without loss of generality, by Remark 2.2, we may assume that all the underlying space used in the
present lemma are $\sigma$-finite. Hence, from \cite[Corollary 7.2.11]{gra2} (multi-linear interpolation)
with the indexes therein taken as follows:
$$m=2=q_0,\ q_1=\fz,\ p_{0,1}=2=p_{0,2},\ p_{1,1}=r,\ p_{1,2}=r',$$
we easily deduce the desired result.
\end{proof}

\begin{definition}
Let $p,\,q,\,r\in[1,\fz]$. A bilinear operator $T$ is said to be of restricted weak type $(p,q;r)$,
if there exists a positive constant $C$ such that, for every measurable subsets $U,\,V$ of $G$
with finite measure,
$$\sup_{t\in(0,\fz)}\lf\{t^{1/r}\lf[T\lf(\mathbf{1}_U,\mathbf{1}_V\r)\r]^{**}(t)\r\}
\le C\mu(U)^{1/p}\mu(V)^{1/q},$$
where the symbol $[\cdots]^{**}$ denotes a maximal function defined as in Remark \ref{2r1}(iii).
\end{definition}

The succeeding Propositions \ref{3p1} and \ref{3p2} come, respectively, from
\cite[p.\,256, Proposition 7.2 and p.\,217, Proposition 4.2]{bs88}; see also \cite{gra,gra2}.

\begin{proposition}\label{3p1}
Let $p,\,q\in[1,\fz)$ and $r\in(1,\fz]$. Then $T$ is of restricted weak type $(p,q;r)$ if and only if
it extends uniquely to a bounded bilinear operator
$$T:\ L^{p,1}(G)\times L^{q,1}(G)\rightarrow L^{r,\fz}(G).$$
\end{proposition}

\begin{proposition}\label{3p2}
Let $p\in(0,\fz]$ and $0<q\le r\le\fz$. Then there exists a positive constant $C$ such that, for any
$\mu$-a.e. finite measurable function $f$,
$$\|f\|_{L^{p,r}(G)}\le C\|f\|_{L^{p,q}(G)}.$$
In other words, for any fixed $p$, the Lorentz spaces $L^{p,q}(G)$ increase as the exponent $q$ increases.
\end{proposition}

Using these above, we have the following conclusion.

\begin{lemma}\label{3l2}
The short-time Fourier transform $V:\,(f,g)\mapsto V_gf$ is of
\begin{enumerate}
\item[{\rm(i)}] restricted weak type $(2,2;2)$;
\item[{\rm(ii)}] restricted weak type $(1,\fz;\fz)$;
\item[{\rm(iii)}] restricted weak type $(\fz,1;\fz)$.
\end{enumerate}
\end{lemma}

\begin{proof}
For (i), by Lemma \ref{3l1} with $p=2=q$ and Proposition \ref{3p2}, we know that, for any
$(f,g)\in L^{2,1}(G)\times L^{2,1}(G)$,
\begin{align*}
\lf\|V_gf\r\|_{L^{2,\fz}(G\times\widehat{G})}
&\ls \lf\|V_gf\r\|_{L^{2,2}(G\times\widehat{G})}
=\lf\|V_gf\r\|_{L^{2}(G\times\widehat{G})}\\
&=\|f\|_{L^{2}(G)}\|g\|_{L^{2}(G)}
\ls\|f\|_{L^{2,1}(G)}\|g\|_{L^{2,1}(G)}.
\end{align*}
This, together with Proposition \ref{3p1}, implies the validity of (i).

To deal with (ii), let $U,\,V$ be arbitrary measurable subsets of $G$ with finite measure. Combining
\eqref{2e2'}, \eqref{2e1} and Lemma \ref{3l1} with $p=\fz=q$, we conclude that
\begin{align}\label{3e1}
\sup_{t\in(0,\fz)}\lf\{t^{1/\fz}\lf[V_{\mathbf{1}_V}\mathbf{1}_U\r]^{**}(t)\r\}
&=\sup_{t\in(0,\fz)}\lf\{\frac1t\int_0^t \lf[V_{\mathbf{1}_V}\mathbf{1}_U\r]^*(u)\,du\r\}\\
&\le\sup_{u\in(0,\fz)}\lf[V_{\mathbf{1}_V}\mathbf{1}_U\r]^*(u)\noz\\
&=\lf\|V_{\mathbf{1}_V}\mathbf{1}_U\r\|_{L^{\fz}(G\times\widehat{G})}\noz\\
&\le\|\mathbf{1}_U\|_{L^{1}(G)}\|\mathbf{1}_V\|_{L^{\fz}(G)}=\mu(U).\noz
\end{align}
Thus, (ii) is true.

Turn to the term (iii), similar to \eqref{3e1}, applying now Lemma \ref{3l1} with $p=1$ and $q=\fz$,
we get
\begin{align*}
\sup_{t\in(0,\fz)}\lf\{t^{1/\fz}\lf[V_{\mathbf{1}_V}\mathbf{1}_U\r]^{**}(t)\r\}
\le\|\mathbf{1}_U\|_{L^{\fz}(G)}\|\mathbf{1}_V\|_{L^{1}(G)}=\mu(V),
\end{align*}
which completes the proof of (iii) and hence of Lemma \ref{3l2}.
\end{proof}

Furthermore, the next Hardy's inequality and Young's inequality are from
\cite[p.\,124, Lemma 3.9]{bs88} and \cite[Theorem 2.12]{gra}, respectively.

\begin{proposition}[Hardy's inequality]\label{3p3}
Let $\phi$ be a nonnegative measurable real function on $(0,\fz)$, $\dz\in(-\fz,1)$ and $q\in[1,\fz]$.
Then
\begin{align*}
\lf\{\int_0^\fz\lf[t^{\dz-1}\int_0^t\phi(u)\,du\r]^q\frac{dt}{t}\r\}^{1/q}
\le\frac1{1-\dz}\lf\{\int_0^\fz\lf[t^{\dz}\phi(t)\r]^q\frac{dt}{t}\r\}^{1/q}
\end{align*}
and
\begin{align*}
\lf\{\int_0^\fz\lf[t^{1-\dz}\int_t^\fz\phi(u)\,du\r]^q\frac{dt}{t}\r\}^{1/q}
\le\frac1{1-\dz}\lf\{\int_0^\fz\lf[t^{1-\dz}\phi(t)\r]^q\frac{dt}{t}\r\}^{1/q}
\end{align*}
with the usual modification made when $q=\fz$.
\end{proposition}

\begin{proposition}[Young's inequality]\label{3p4}
Let $u,\,v,\,w\in[1,\fz]$ such that
$$\frac1u+\frac1v=1+\frac1w.$$
Then, for any $f\in L^u(\mathbb{R}_+,\frac{dx}x)$ and $g\in L^v(\mathbb{R}_+,\frac{dx}x)$,
the convolution
\begin{align}\label{3e2}
\lf(f\ast_{\mathbb{R}_+}g\r)(x):=\int_0^\fz f(y)g\lf(\frac xy\r)\,\frac{dy}y
\end{align}
exists a.e. and satisfies
\begin{align*}
\lf\|f\ast_{\mathbb{R}_+}g\r\|_{L^w(\mathbb{R}_+,\frac{dx}x)}
\le\|f\|_{L^u(\mathbb{R}_+,\frac{dx}x)}\|g\|_{L^v(\mathbb{R}_+,\frac{dx}x)}.
\end{align*}
\end{proposition}

We also need the help of Calder\'{o}n operators.

\begin{definition}\label{3d1}
Let $K\in\nn$ and $u_k,\,v_k,\,w_k\in[1,\fz]$ for any $k\in[1,K]\cap\nn$. Denote by $\eta$ the set
of points:
\begin{align*}
\eta=\lf\{\lf(\frac1{u_k},\frac1{v_k},\frac1{w_k}\r)\r\}_{k\in[1,K]\cap\nn}.
\end{align*}
For any measurable function $f$ and $g$, the corresponding Calder\'{o}n operator is defined as
$$S_\eta(f,g)(t):=\int_0^\fz\int_0^\fz f(r)g(s)\min_{k\in[1,K]\cap\nn}
\lf\{\frac{r^{1/u_k}s^{1/v_k}}{t^{1/w_k}}\r\}\,\frac{dr}r\frac{ds}s,\quad\forall\,t\in(0,\fz).$$
\end{definition}

Recall that a \emph{simple function} is a finite linear combination of characteristic functions of sets
of finite measure. The succeeding Proposition \ref{3p5} is just \cite[p.\,257, Theorem 7.4]{bs88}.

\begin{proposition}\label{3p5}
With the same symbols as in Definition \ref{3d1}, assume that a bilinear operator $T$ is of restricted
weak type $(u_k,v_k;w_k)$ for each $k\in[1,K]\cap\nn$, that is, there exists a sequence of positive
constants $\{C_k\}_{k\in[1,K]\cap\nn}$ such that, for any measurable subsets $U,\,V$ of $G$
with finite measure,
$$\sup_{t\in(0,\fz)}\lf\{t^{1/w_k}\lf[T\lf(\mathbf{1}_U,\mathbf{1}_V\r)\r]^{**}(t)\r\}
\le C_k\mu(U)^{1/u_k}\mu(V)^{1/v_k}.$$
Then, for any simple functions $f$ and $g$,
$$\lf[T(f,g)\r]^*(t)\le\max_{k\in[1,K]\cap\nn}\{C_k\}S_\eta\lf(f^*,g^*\r)(t),\quad\forall\,t\in(0,\fz).$$
\end{proposition}

\begin{remark}\label{3r1}
Observe that the short-time Fourier transform $V_gf$ is linear in $f$ but conjugate linear in $g$;
however, the aforementioned results on bilinear operators $T$ hold also in this case,
since $(\overline{g})^*=g^*$ for any measurable function $g$.
\end{remark}

Now, we take $\eta$ in Definition \ref{3d1} as
\begin{align}\label{3e3}
\eta=\lf\{\lf(\frac12,\frac12,\frac12\r),\lf(1,0,0\r),\lf(0,1,0\r)\r\}
\end{align}
and consider the corresponding Calder\'{o}n operator
\begin{align}\label{3e4}
S_\eta(f,g)(t)=\int_0^\fz\int_0^\fz f(r)g(s)\min
\lf\{\sqrt{\frac{rs}{t}},r,s\r\}\,\frac{dr}r\frac{ds}s,\quad\forall\,t\in(0,\fz).
\end{align}
For this operator, we have an useful estimation on Lorentz spaces as follows, which plays a key role
in the proof of Theorem \ref{t1} and is also independent interest.

\begin{lemma}\label{3l3}
Let $q\in(2,\fz]$, $p\in[q',q]$ with $p\neq2$, and $u,\,v,\,w$ be as in Proposition \ref{3p4}.
\begin{enumerate}
\item[{\rm(i)}] If $q\in(2,\fz)$,
then there exists a positive constant $C$ such that, for any $(f,g)\in L^{p',u}(G)\times L^{p,v}(G)$,
\begin{align}\label{3e5}
\lf\{\int_0^\fz\lf[t^{1/q}S_\eta\lf(f^*,g^*\r)(t)\r]^w\frac{dt}{t}\r\}^{1/w}
\le C\|f\|_{L^{p',u}(G)}\|g\|_{L^{p,v}(G)}
\end{align}
with the usual modification made when $w=\fz$.
\item[{\rm(ii)}] If $q=\fz=w$,
then there exists a positive constant $\wz C$ such that, for any $(f,g)\in L^{p',u}(G)\times L^{p,v}(G)$,
\begin{align*}
\lf\|S_\eta\lf(f^*,g^*\r)\r\|_{L^\fz(\mathbb{R}_+)}
\le \wz C\|f\|_{L^{p',u}(G)}\|g\|_{L^{p,v}(G)}.
\end{align*}
\end{enumerate}
Here $\eta$ and $S_\eta$ are as in \eqref{3e3} and \eqref{3e4}, respectively.
\end{lemma}

\begin{proof}
We prove the present lemma by three steps.

\textbf{Step 1.} In this step, we show Lemma \ref{3l3}(i) for the case $w\in[1,\fz)$. In this case,
the indexes $u,\,v$ must also belong to $[1,\fz)$, since $u,\,v,\,w\in[1,\fz]$ and $\frac1u+\frac1v=1+\frac1w.$
Obviously, for any $(f,g)\in L^{p',u}(G)\times L^{p,v}(G)$,
\begin{align*}
&\lf\{\int_0^\fz\lf[t^{1/q}S_\eta\lf(f^*,g^*\r)(t)\r]^w\frac{dt}{t}\r\}^{1/w}\\
&\hs\le \lf\{\int_0^1\lf[t^{1/q}S_\eta\lf(f^*,g^*\r)(t)\r]^w\frac{dt}{t}\r\}^{1/w}
+\lf\{\int_1^\fz\lf[t^{1/q}S_\eta\lf(f^*,g^*\r)(t)\r]^w\frac{dt}{t}\r\}^{1/w}.\noz\\
&\hs=:{\rm J_1}+{\rm J_2}.\noz
\end{align*}
To deal with ${\rm J_1}$ and ${\rm J_2}$, we consider the following five domains of
$\{(r,s):\ r,\,s\in[0,\fz)\}$:
\begin{center}
\begin{tabular}{cc}
\begin{tikzpicture}[scale=0.8]
    % 坐标轴
    \draw[->, thick] (0,0) -- (6,0) node[right] {$r$};
    \draw[->, thick] (0,0) -- (0,4) node[above] {$s$};

    % 直线
    \draw[red, thick] (0,0) -- (2.5,4) node[right] {$s = tr$};
    \draw[blue, thick] (0,0) -- (5,2) node[right] {$s = \frac1tr$};

    % 区域标签
    \node at (0.5,2) {$D_1$};
    \node at (2.5,2) {$D_2$};
    \node at (4,0.8) {$D_3$};
\end{tikzpicture}

\hspace{2cm}

\begin{tikzpicture}[scale=0.8]
    % 坐标轴
    \draw[->, thick] (0,0) -- (4,0) node[right] {$r$};
    \draw[->, thick] (0,0) -- (0,4) node[above] {$s$};

    % 直线
    \draw[cyan, thick] (0,0) -- (3.5,3.5) node[right] {$s = r$};

    % 区域标签
    \node at (1.25,2.5) {$D_4$};
    \node at (2.75,1.05) {$D_5$};
\end{tikzpicture}\\
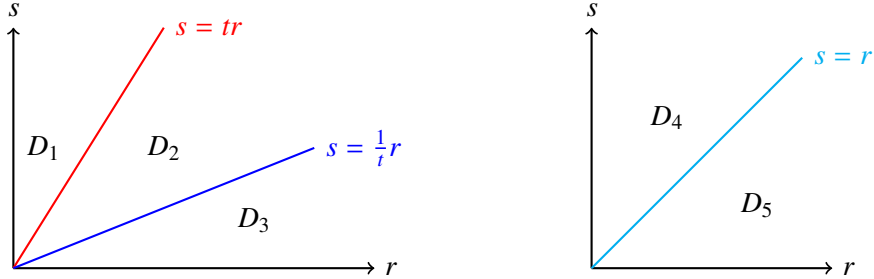
\multicolumn{1}{c}{\scriptsize{\textbf{Figure 1.} Definition of the domains $D_1,\ D_2,\ D_3,\ D_4,\ D_5$.}}
\end{tabular}
\end{center}
Then we have
\begin{enumerate}
\item[$\blacktriangleright$] If $t\in(1,\fz)$, then
\begin{align}\label{3e7}
\min\lf\{\sqrt{\frac{rs}{t}},r,s\r\}=
\begin{cases}
r
&\text{when}\ (r,s)\in D_1,\\
\sqrt{\frac{rs}{t}}
&\text{when}\ (r,s)\in D_2,\\
s
&\text{when}\ (r,s)\in D_3.
\end{cases}
\end{align}
\item[$\blacktriangleright$] If $t\in(0,1]$, then
\begin{align}\label{3e8}
\min\lf\{\sqrt{\frac{rs}{t}},r,s\r\}=
\begin{cases}
r
&\text{when}\ (r,s)\in D_4,\\
s
&\text{when}\ (r,s)\in D_5.
\end{cases}
\end{align}
\end{enumerate}

For ${\rm J_2}$, by the definition of $S_\eta$ in \eqref{3e4} and \eqref{3e7}, we find that
\begin{align*}
{\rm J_2}&=\lf\{\int_1^\fz\lf[t^{1/q}\int_0^\fz\int_0^\fz f^*(r)g^*(s)\min
\lf\{\sqrt{\frac{rs}{t}},r,s\r\}\,\frac{dr}r\frac{ds}s\r]^w\frac{dt}{t}\r\}^{1/w}\\
&=\lf\{\int_1^\fz\lf[t^{1/q}\iint_{D_1} f^*(r)g^*(s)r\,\frac{dr}r\frac{ds}s
+t^{1/q}\iint_{D_2} f^*(r)g^*(s)\sqrt{\frac{rs}{t}}\,\frac{dr}r\frac{ds}s\r.\r.\noz\\
&\qquad\qquad\lf.\lf.+t^{1/q}\iint_{D_3} f^*(r)g^*(s)s\,\frac{dr}r\frac{ds}s\r]^w\frac{dt}{t}\r\}^{1/w}\noz\\
&\le3^{1-1/w}\lf(\lf\{\int_1^\fz\lf[t^{1/q}
\iint_{D_1} f^*(r)g^*(s)\,dr\frac{ds}s\r]^w\frac{dt}{t}\r\}^{1/w}\r.\noz\\
&\qquad\qquad+\lf\{\int_1^\fz\lf[t^{1/q}
\iint_{D_2} f^*(r)g^*(s)\sqrt{\frac{rs}{t}}\,\frac{dr}r\frac{ds}s\r]^w\frac{dt}{t}\r\}^{1/w}\noz\\
&\qquad\qquad\lf.+\lf\{\int_1^\fz\lf[t^{1/q}
\iint_{D_3} f^*(r)g^*(s)\,\frac{dr}rds\r]^w\frac{dt}{t}\r\}^{1/w}\r).\noz
\end{align*}
Similarly, the corresponding estimation for ${\rm J_1}$ can be given with the help of \eqref{3e8}.
Overall, we get that
\begin{align}\label{3e9}
&\lf\{\int_0^\fz\lf[t^{1/q}S_\eta\lf(f^*,g^*\r)(t)\r]^w\frac{dt}{t}\r\}^{1/w}\\
&\quad\le3^{1-1/w}\lf(\lf\{\int_1^\fz\lf[t^{1/q}
\iint_{D_1} f^*(r)g^*(s)\,dr\frac{ds}s\r]^w\frac{dt}{t}\r\}^{1/w}\r.\noz\\
&\quad\qquad\qquad+\lf\{\int_1^\fz\lf[t^{1/q}
\iint_{D_2} f^*(r)g^*(s)\sqrt{\frac{rs}{t}}\,\frac{dr}r\frac{ds}s\r]^w\frac{dt}{t}\r\}^{1/w}\noz\\
&\quad\qquad\qquad+\lf\{\int_1^\fz\lf[t^{1/q}
\iint_{D_3} f^*(r)g^*(s)\,\frac{dr}rds\r]^w\frac{dt}{t}\r\}^{1/w}\noz\\
&\quad\qquad\qquad+\lf\{\int_0^1\lf[t^{1/q}
\iint_{D_4} f^*(r)g^*(s)\,dr\frac{ds}s\r]^w\frac{dt}{t}\r\}^{1/w}\noz\\
&\quad\qquad\qquad\lf.+\lf\{\int_0^1\lf[t^{1/q}
\iint_{D_5} f^*(r)g^*(s)\,\frac{dr}rds\r]^w\frac{dt}{t}\r\}^{1/w}\r)\noz\\
&\quad=:3^{1-1/w}\lf({\rm J_{2,1}}+{\rm J_{2,2}}+{\rm J_{2,3}}+{\rm J_{1,1}}+{\rm J_{1,2}}\r).\noz
\end{align}
We next estimate these five terms appearing in \eqref{3e9} one by one.

\textbf{Substep 1.1} (Estimation of ${\rm J_{2,1}}$). Note that $D_1=\{(r,s):\ s\in[0,\fz),\ r\in[0,s/t)\}$.
Then it is easy to see that, for any $p\in[1,q]$,
\begin{align}\label{3e10}
{\rm J_{2,1}}&\le\lf\{\int_1^\fz\lf[t^{1/p}
\int_0^\fz g^*(s)\lf(\int_0^{s/t}f^*(r)\,dr\r)\,\frac{ds}s\r]^w\frac{dt}{t}\r\}^{1/w}\\
&\le\lf\{\int_0^\fz\lf[\int_0^\fz
s^{1/p}g^*(s)\lf(\frac ts\r)^{1/p}\lf(\int_0^{s/t}f^*(r)\,dr\r)\,\frac{ds}s\r]^w\frac{dt}{t}\r\}^{1/w}.\noz
\end{align}
For any $z\in\mathbb{R}_+$, let
\begin{align}\label{3e11}
F(z):=z^{1/p}\int_0^{1/z}f^*(r)\,dr\quad {\rm and}\quad H(z):=z^{1/p}g^*(z).
\end{align}
From this, \eqref{3e10}, the definition of convolutions on $\mathbb{R}_+$ (see \eqref{3e2})
and Proposition \ref{3p4}, we infer that
\begin{align}\label{3e11'}
{\rm J_{2,1}}
&\le\lf\{\int_0^\fz\lf[\lf(H\ast_{\mathbb{R}_+}F\r)(t)\r]^w\frac{dt}{t}\r\}^{1/w}\\
&\le\|H\|_{L^v(\mathbb{R}_+,\frac{dt}t)}\|F\|_{L^u(\mathbb{R}_+,\frac{dt}t)}\noz\\
&=\lf\{\int_0^\fz\lf[t^{1/p}g^*(t)\r]^v\frac{dt}{t}\r\}^{1/v}
\lf\{\int_0^\fz\lf[t^{1/p}\int_0^{1/t}f^*(r)\,dr\r]^u\frac{dt}{t}\r\}^{1/u}\noz\\
&=\|g\|_{L^{p,v}(G)}\lf\{\int_0^\fz\lf[t^{-1/p}\int_0^{t}f^*(r)\,dr\r]^u\frac{dt}{t}\r\}^{1/u},\noz
\end{align}
where we used \eqref{2e1} and made a change of variable $t$ to obtain the last equality. This, combined
with Hardy's inequality (see Proposition \ref{3p3}) with $\dz$ therein taken as $1-1/p$ and \eqref{2e1} again,
further implies that
\begin{align}\label{3e12}
{\rm J_{2,1}}
&\le p\|g\|_{L^{p,v}(G)}\lf\{\int_0^\fz\lf[t^{1-1/p}f^*(t)\r]^u\frac{dt}{t}\r\}^{1/u}\\
&=p\|g\|_{L^{p,v}(G)}\|f\|_{L^{p',u}(G)}\noz
\end{align}
holds true for any $q\in(2,\fz)$, $p\in[1,q]$ and $u,\,v,\,w\in[1,\fz)$ satisfying $\frac1u+\frac1v=1+\frac1w$,
which is desired.

\textbf{Substep 1.2} (Estimation of ${\rm J_{2,2}}$). Observe that $D_2=\{(r,s):\ r\in[0,\fz),\ s\in(r/t,tr]\}$.
Then, for any $p\in[1,q]$, we have
\begin{align*}
{\rm J_{2,2}}&\le\lf\{\int_1^\fz\lf[t^{1/p-1/2}\int_0^\fz
f^*(r)\lf(\int_{r/t}^{tr}g^*(s)\,\frac{ds}{\sqrt{s}}\r)\,\frac{dr}{\sqrt{r}}\r]^w\frac{dt}{t}\r\}^{1/w}\\
&\le\lf\{\int_0^\fz\lf[t^{1/p-1/2}\int_0^\fz
f^*(r)\lf(\int_{r/t}^{\fz}g^*(s)\,\frac{ds}{\sqrt{s}}\r)\,\frac{dr}{\sqrt{r}}\r]^w\frac{dt}{t}\r\}^{1/w}\\
&=\lf\{\int_0^\fz\lf[\int_0^\fz r^{1/p}f^*(r)\lf(\frac tr\r)^{1/p-1/2}
\lf(\int_{r/t}^{\fz}g^*(s)\,\frac{ds}{\sqrt{s}}\r)\,\frac{dr}{r}\r]^w\frac{dt}{t}\r\}^{1/w}.
\end{align*}
For any $z\in\mathbb{R}_+$, now let
\begin{align}\label{3e14}
\wz F(z):=z^{1/p}f^*(z)\quad {\rm and}\quad \wz H(z):=z^{1/p-1/2}\int_{1/z}^\fz g^*(s)\,\frac{ds}{\sqrt{s}}.
\end{align}
By this and an argument similar to that used in \eqref{3e11'}, we obtain
\begin{align*}
{\rm J_{2,2}}
&\le\lf\{\int_0^\fz\lf[\lf(\wz F\ast_{\mathbb{R}_+}\wz H\r)(t)\r]^w\frac{dt}{t}\r\}^{1/w}\\
&\le\lf\|\wz F\r\|_{L^u(\mathbb{R}_+,\frac{dt}t)}\lf\|\wz H\r\|_{L^v(\mathbb{R}_+,\frac{dt}t)}\\
&=\lf\{\int_0^\fz\lf[t^{1/p}f^*(t)\r]^u\frac{dt}{t}\r\}^{1/u}
\lf\{\int_0^\fz\lf[t^{1/p-1/2}\int_{1/t}^\fz g^*(s)\,\frac{ds}{\sqrt{s}}\r]^v\frac{dt}{t}\r\}^{1/v}\\
&=\|f\|_{L^{p,u}(G)}
\lf\{\int_0^\fz\lf[t^{1/2-1/p}\int_{t}^\fz g^*(s)\,\frac{ds}{\sqrt{s}}\r]^v\frac{dt}{t}\r\}^{1/v}.
\end{align*}
For any $q\in(2,\fz)$ and $p\in(2,q]$, using Hardy's inequality (see Proposition \ref{3p3})
with $\dz$ therein taken as $1/2+1/p$, we conclude that
\begin{align}\label{3e15}
{\rm J_{2,2}}
&\le\|f\|_{L^{p,u}(G)}
\lf\{\int_0^\fz\lf[t^{1/2-1/p}\int_{t}^\fz \sqrt{s}g^*(s)\,\frac{ds}{s}\r]^v\frac{dt}{t}\r\}^{1/v}\\
&\le\|f\|_{L^{p,u}(G)}\frac1{1-(1/2+1/p)}
\lf\{\int_0^\fz\lf[t^{1/2-1/p}\sqrt{t}g^*(t)\r]^v\frac{dt}{t}\r\}^{1/v}\noz\\
&=\frac{2p}{p-2}\|f\|_{L^{p,u}(G)}\|g\|_{L^{p',v}(G)}\noz
\end{align}
is valid for any $u,\,v,\,w\in[1,\fz)$ satisfying $\frac1u+\frac1v=1+\frac1w$.

\textbf{Substep 1.3} (Estimation of ${\rm J_{2,3}}$). Notice that $D_3=\{(r,s):\ r\in[0,\fz),\ s\in[0,r/t)\}$
and $D_1=\{(r,s):\ s\in[0,\fz),\ r\in[0,s/t)\}$. Interchanging $r$ and $s$, it follows that
\begin{align}\label{3e16}
{\rm J_{2,3}}=&\lf\{\int_1^\fz\lf[t^{1/q}\iint_{D_3} f^*(r)g^*(s)\,\frac{dr}rds\r]^w\frac{dt}{t}\r\}^{1/w}\\
&=\lf\{\int_1^\fz\lf[t^{1/q}\iint_{D_1} f^*(s)g^*(r)\,\frac{ds}s dr\r]^w\frac{dt}{t}\r\}^{1/w}.\noz
\end{align}
Applying \eqref{3e12} with $f$ and $g$ interchanged, we get
\begin{align*}
\lf\{\int_1^\fz\lf[t^{1/q}\iint_{D_1} f^*(s)g^*(r)\,\frac{ds}s dr\r]^w\frac{dt}{t}\r\}^{1/w}
\le p\|f\|_{L^{p,v}(G)}\|g\|_{L^{p',u}(G)}.
\end{align*}
From this and \eqref{3e16}, we deduce that, for any $q\in(2,\fz)$, $p\in[1,q]$ and $u,\,v,\,w\in[1,\fz)$
satisfying $\frac1u+\frac1v=1+\frac1w$,
\begin{align}\label{3e17}
{\rm J_{2,3}}\le p\|f\|_{L^{p,v}(G)}\|g\|_{L^{p',u}(G)}.
\end{align}

\textbf{Substep 1.4} (Estimation of ${\rm J_{1,1}}$). Since $D_4=\{(r,s):\ s\in[0,\fz),\ r\in[0,s)\}$,
by \eqref{2e2'}, we find that, for any $q\in(2,\fz)$,
\begin{align*}
{\rm J_{1,1}}&=
\lf\{\int_0^1\lf[t^{1/q}\int_0^\fz g^*(s)\lf(\frac1s\int_0^s f^*(r)\,dr\r)ds\r]^w\frac{dt}{t}\r\}^{1/w}\\
&=\lf\{\int_0^1\lf[t^{1/q}\int_0^\fz g^*(s)f^{**}(s)\,ds\r]^w\frac{dt}{t}\r\}^{1/w}\\
&=\lf(\int_0^1t^{w/q}\,\frac{dt}{t}\r)^{1/w}\int_0^\fz g^*(s)f^{**}(s)\,ds.
\end{align*}
From this, the H\"{o}lder inequality and Remark \ref{2r1}(iii), we infer that, for any $q\in(2,\fz)$,
$p\in(1,\fz]$ and $\kappa\in[1,\fz]$,
\begin{align}\label{3e18}
{\rm J_{1,1}}
&=\lf(\frac qw\r)^{1/w}\int_0^\fz \lf[s^{1/p'}g^*(s)\r]\lf[s^{1/p}f^{**}(s)\r]\,\frac{ds}s\\
&\ls\lf\|(\cdot)^{1/p'}g^*(\cdot)\r\|_{L^{\kappa'}(\mathbb{R}_+,\frac{ds}s)}
\lf\|(\cdot)^{1/p}f^{**}(\cdot)\r\|_{L^{\kappa}(\mathbb{R}_+,\frac{ds}s)}\noz\\
&=\lf\{\int_0^\fz \lf[s^{1/p'}g^*(s)\r]^{\kappa'}\,\frac{ds}s\r\}^{1/\kappa'}
\lf\{\int_0^\fz \lf[s^{1/p}f^{**}(s)\r]^{\kappa}\,\frac{ds}s\r\}^{1/\kappa}\noz\\
&\ls\|g\|_{L^{p',\kappa'}(G)}\|f\|_{L^{p,\kappa}(G)}.\noz
\end{align}
For any $u,\,v,\,w\in[1,\fz)$ satisfying $\frac1u+\frac1v=1+\frac1w$, take $\kappa=u$.
Then $v<\kappa'$ and hence
$$\|g\|_{L^{p',\kappa'}(G)}\ls\|g\|_{L^{p',v}(G)}$$
by Proposition \ref{3p2}. This, together with \eqref{3e18}, further implies that
\begin{align}\label{3e19}
{\rm J_{1,1}}&=\lf\{\int_0^1\lf[t^{1/q}\iint_{D_4} f^*(r)g^*(s)\,dr\frac{ds}s\r]^w\frac{dt}{t}\r\}^{1/w}\\
&\ls\|g\|_{L^{p',v}(G)}\|f\|_{L^{p,u}(G)}\noz
\end{align}
holds true for any $q\in(2,\fz)$, $p\in(1,\fz]$ and $u,\,v,\,w\in[1,\fz)$ satisfying
$\frac1u+\frac1v=1+\frac1w$.

\textbf{Substep 1.5} (Estimation of ${\rm J_{1,2}}$). Observe that $D_5=\{(r,s):\ r\in[0,\fz),\ s\in[0,r)\}$
and $D_4=\{(r,s):\ s\in[0,\fz),\ r\in[0,s)\}$. Then, as in Substep 1.3, interchanging $r$ and $s$, we obtain
\begin{align*}
{\rm J_{1,2}}=&\lf\{\int_0^1\lf[t^{1/q}\iint_{D_5} f^*(r)g^*(s)\,\frac{dr}rds\r]^w\frac{dt}{t}\r\}^{1/w}\\
&=\lf\{\int_0^1\lf[t^{1/q}\iint_{D_4} f^*(s)g^*(r)\,\frac{ds}sdr\r]^w\frac{dt}{t}\r\}^{1/w}.\noz
\end{align*}
Using \eqref{3e19} with $f$ and $g$ interchanged, we get
\begin{align*}
\lf\{\int_0^1\lf[t^{1/q}\iint_{D_4} f^*(s)g^*(r)\,\frac{ds}sdr\r]^w\frac{dt}{t}\r\}^{1/w}
\ls\|f\|_{L^{p',v}(G)}\|g\|_{L^{p,u}(G)}.
\end{align*}
From this and \eqref{3e16}, it follows that, for any $q\in(2,\fz)$, $p\in(1,\fz]$ and $u,\,v,\,w\in[1,\fz)$ satisfying $\frac1u+\frac1v=1+\frac1w$,
\begin{align}\label{3e21}
{\rm J_{1,2}}\ls\|f\|_{L^{p',v}(G)}\|g\|_{L^{p,u}(G)}.
\end{align}

Combining \eqref{3e9}, \eqref{3e12}, \eqref{3e15}, \eqref{3e17}, \eqref{3e19} and \eqref{3e21}, we conclude
that \eqref{3e5} is true for any $q\in(2,\fz)$, $p\in(2,q]$ and $u,\,v,\,w\in[1,\fz)$ satisfying
$\frac1u+\frac1v=1+\frac1w$. By this and the fact that
\begin{align*}
S_\eta(f,g)(t)=S_\eta(g,f)(t),\quad\forall\,t\in(0,\fz)
\end{align*}
(see \eqref{3e4}), it is easy to see that \eqref{3e5} is also true for $p\in[q',2)$.  At this point,
the proof of Lemma \ref{3l3}(i) for the case $w\in[1,\fz)$ is completed.

\textbf{Step 2.} In this step, we prove Lemma \ref{3l3}(i) for the case $w=\fz$. The idea to achieve this
is same as that used in Step 1. For the sake of completeness, details are given here. By an argument
similar to that used in \eqref{3e9}, we have
\begin{align}\label{3e22}
\sup_{t\in(0,\fz)}\lf[t^{1/q}S_\eta\lf(f^*,g^*\r)(t)\r]
&\ls\sup_{t\in(1,\fz)}\lf[t^{1/q}
\iint_{D_1} f^*(r)g^*(s)\,dr\frac{ds}s\r]\\
&\quad+\sup_{t\in(1,\fz)}\lf[t^{1/q}
\iint_{D_2} f^*(r)g^*(s)\sqrt{\frac{rs}{t}}\,\frac{dr}r\frac{ds}s\r]\noz\\
&\quad+\sup_{t\in(1,\fz)}\lf[t^{1/q}
\iint_{D_3} f^*(r)g^*(s)\,\frac{dr}rds\r]\noz\\
&\quad+\sup_{t\in(0,1)}\lf[t^{1/q}
\iint_{D_4} f^*(r)g^*(s)\,dr\frac{ds}s\r]\noz\\
&\quad+\sup_{t\in(0,1)}\lf[t^{1/q}
\iint_{D_5} f^*(r)g^*(s)\,\frac{dr}rds\r]\noz\\
&=:{\rm \wz J_{2,1}}+{\rm \wz J_{2,2}}+{\rm \wz J_{2,3}}+{\rm \wz J_{1,1}}+{\rm \wz J_{1,2}}.\noz
\end{align}
We next estimate these five terms appearing in \eqref{3e22} one by one.

For ${\rm \wz J_{2,1}}$, let $F$ and $H$ be two functions as in \eqref{3e11}. Then, following the proof
of \eqref{3e11'}, we find that, for any $u,\,v\in[1,\fz]$ satisfying $\frac1u+\frac1v=1$,
\begin{align*}
{\rm \wz J_{2,1}}
\le\lf\|H\ast_{\mathbb{R}_+}F\r\|_{L^\fz(\mathbb{R}_+)}
\le\|H\|_{L^v(\mathbb{R}_+,\frac{dt}t)}\|F\|_{L^u(\mathbb{R}_+,\frac{dt}t)}.
\end{align*}
Notice that only one between $u$ and $v$ can be $\fz$. If $u=\fz$, then
\begin{align*}
\|F\|_{L^u(\mathbb{R}_+,\frac{dt}t)}&=\|F\|_{L^\fz(\mathbb{R}_+)}\\
&=\sup_{t\in(0,\fz)}\lf[t^{1/p}\int_0^{1/t}f^*(r)\,dr\r]\\
&=\sup_{t\in(0,\fz)}\lf[t^{-1/p}\int_0^{t}f^*(r)\,dr\r]\\
&=\sup_{t\in(0,\fz)}\lf[t^{1-1/p}\frac1t\int_0^{t}f^*(r)\,dr\r]\\
&=\sup_{t\in(0,\fz)}\lf[t^{1/p'}f^{**}(t)\r]
\sim\|f\|_{L^{p',\fz}(G)}=\|f\|_{L^{p',u}(G)},
\end{align*}
where we used Remark \ref{2r1}(iii) in the last line. If $v=\fz$, then
\begin{align*}
\|H\|_{L^v(\mathbb{R}_+,\frac{dt}t)}&=\|H\|_{L^\fz(\mathbb{R}_+)}
=\sup_{z\in(0,\fz)}\lf[z^{1/p}g^{*}(z)\r]\\
&=\|g\|_{L^{p,\fz}(G)}=\|g\|_{L^{p,v}(G)}.
\end{align*}
Based on these two points,
the remainder computations for ${\rm \wz J_{2,1}}$ are exactly the same as before.

To deal with ${\rm \wz J_{2,2}}$, let $\wz F$ and $\wz H$ be two functions as in \eqref{3e14}.
Repeating the estimation of ${\rm J_{2,2}}$ in Substep 1.2 with some slight modifications,
we know that, for any $u,\,v\in[1,\fz]$ satisfying $\frac1u+\frac1v=1$,
\begin{align*}
{\rm \wz J_{2,2}}
\le\lf\|\wz F\ast_{\mathbb{R}_+\wz H}\r\|_{L^\fz(\mathbb{R}_+)}
\le\lf\|\wz F\r\|_{L^u(\mathbb{R}_+,\frac{dt}t)}\lf\|\wz H\r\|_{L^v(\mathbb{R}_+,\frac{dt}t)}.
\end{align*}
As for ${\rm \wz J_{2,1}}$, if $v=\fz$, then
\begin{align*}
\lf\|\wz H\r\|_{L^v(\mathbb{R}_+,\frac{dt}t)}&=\lf\|\wz H\r\|_{L^\fz(\mathbb{R}_+)}\\
&=\sup_{t\in(0,\fz)}\lf[t^{1/p-1/2}\int_{1/t}^\fz g^*(s)\,\frac{ds}{\sqrt{s}}\r]\\
&=\sup_{t\in(0,\fz)}\lf[t^{1/p-1/2}\int_{1/t}^\fz s^{1/p'}g^*(s)\,\frac{ds}{s^{1/p'+1/2}}\r]\\
&\le\sup_{t\in(0,\fz)}\lf\{t^{1/p-1/2}\sup_{s\in(0,\fz)}\lf[s^{1/p'}g^*(s)\r]
\int_{1/t}^\fz\,\frac{ds}{s^{1/p'+1/2}}\r\}\\
&=\frac{2p'}{2-p'}\|g\|_{L^{p',\fz}(G)}=\frac{2p}{p-2}\|g\|_{L^{p',v}(G)}.
\end{align*}
If $u=\fz$, then
\begin{align*}
\lf\|\wz F\r\|_{L^u(\mathbb{R}_+,\frac{dt}t)}&=\lf\|\wz F\r\|_{L^\fz(\mathbb{R}_+)}
=\sup_{z\in(0,\fz)}\lf[z^{1/p}f^{*}(z)\r]\\
&=\|f\|_{L^{p,\fz}(G)}=\|f\|_{L^{p,u}(G)}.
\end{align*}
The rest procedure to obtain the desired boundedness inequality for ${\rm \wz J_{2,2}}$ is the same as
for ${\rm J_{2,2}}$ in Substep 1.2.

Observe that the estimates of ${\rm \wz J_{2,3}}$, ${\rm \wz J_{1,1}}$ and ${\rm \wz J_{1,2}}$ do not
present substantial differences with respect to ${\rm J_{2,3}}$, ${\rm J_{1,1}}$ and ${\rm J_{1,2}}$,
respectively. This finishes the proof of Step 2 and hence of Lemma \ref{3l3}(i).

\textbf{Step 3.} In this step, we prove Lemma \ref{3l3}(ii). When $q=\fz=w$, \eqref{3e22} becomes
\begin{align*}
&\sup_{t\in(0,\fz)}\lf[S_\eta\lf(f^*,g^*\r)(t)\r]\\
&\quad\ls\sup_{t\in(1,\fz)}\lf[
\iint_{D_1} f^*(r)g^*(s)\,dr\frac{ds}s\r]+\sup_{t\in(1,\fz)}\lf[
\iint_{D_2} f^*(r)g^*(s)\sqrt{\frac{rs}{t}}\,\frac{dr}r\frac{ds}s\r]\\
&\qquad+\sup_{t\in(1,\fz)}\lf[
\iint_{D_3} f^*(r)g^*(s)\,\frac{dr}rds\r]+\sup_{t\in(0,1)}\lf[
\iint_{D_4} f^*(r)g^*(s)\,dr\frac{ds}s\r]\\
&\qquad+\sup_{t\in(0,1)}\lf[
\iint_{D_5} f^*(r)g^*(s)\,\frac{dr}rds\r].
\end{align*}
Then the computations are exactly the same as before; there is just one point in the
estimate of the first term where something new appears. Indeed, when $p=\fz$, the Hardy's inequality
does not work for the estimation \eqref{3e12}. Instead, by Proposition \ref{3p4}, we have
\begin{align*}
\sup_{t\in(1,\fz)}\lf[\iint_{D_1} f^*(r)g^*(s)\,dr\frac{ds}s\r]
&\le\sup_{t\in(0,\fz)}\lf[\int_0^\fz g^*(s)\lf(\int_0^{s/t}f^*(r)\,dr\r)\,\frac{ds}s\r]\\
&=\lf\|H\ast_{\mathbb{R}_+}F\r\|_{L^\fz(\mathbb{R}_+)}\\
&\le\|H\|_{L^v(\mathbb{R}_+,\frac{dt}t)}\|F\|_{L^u(\mathbb{R}_+,\frac{dt}t)}\\
&=\lf\{\int_0^\fz\lf[g^*(t)\r]^v\frac{dt}{t}\r\}^{1/v}
\lf\{\int_0^\fz\lf[\int_0^{1/t}f^*(r)\,dr\r]^u\frac{dt}{t}\r\}^{1/u},
\end{align*}
where $F$ and $H$ be two functions as in \eqref{3e11} with $p=\fz$.
This, combined with Remark \ref{2r1}(iii), gives us
\begin{align*}
\sup_{t\in(1,\fz)}\lf[\iint_{D_1} f^*(r)g^*(s)\,dr\frac{ds}s\r]
&\le\|g\|_{L^{\fz,v}(G)}\lf\{\int_0^\fz\lf[t\frac1t\int_0^{t}f^*(r)\,dr\r]^u\frac{dt}{t}\r\}^{1/u}\\
&=\|g\|_{L^{\fz,v}(G)}\lf\{\int_0^\fz\lf[t f^{**}(t)\r]^u\frac{dt}{t}\r\}^{1/u}\\
&\sim\|g\|_{L^{\fz,v}(G)}\|g\|_{L^{1,u}(G)}.
\end{align*}
The proof of Lemma \ref{3l3}(ii) and hence of Lemma \ref{3l3} is completed.
\end{proof}

In addition, the following density can be found in \cite[Theorem 1.4.13]{gra}.

\begin{proposition}\label{3p6}
For any $p\in(0,\fz]$ and $q\in(0,\fz)$, the simple functions are dense in $L^{p,q}(G)$.
\end{proposition}

Next, we show Theorem \ref{t1}.

\begin{proof}[\textbf{Proof of Theorem \ref{t1}}]
By Proposition \ref{3p2}, we know that it suffices to prove Theorem \ref{t1} for the indexes satisfying
$\frac1u+\frac1v=1+\frac1w$. To achieve this,
we first let $f$ and $g$ be arbitrary two simple functions. Then, by Lemma \ref{3l2}, Proposition \ref{3p5}
and Remark \ref{3r1}, we find that, for any $q,\,w\in[1,\fz)$,
\begin{align*}
\lf\|V_gf\r\|_{L^{q,w}(G\times\widehat{G})}
\ls\lf\{\int_0^\fz\lf[t^{1/q}S_\eta\lf(f^*,g^*\r)(t)\r]^w\frac{dt}{t}\r\}^{1/w},
\end{align*}
where $\eta$ and $S_\eta$ are as in \eqref{3e3} and \eqref{3e4}, respectively. From this and Lemma \ref{3l3},
we infer that, for any $q\in(2,\fz)$, $p\in[q',q]$ with $p\neq2$, and $u,\,v,\,w\in[1,\fz)$ satisfying
$\frac1u+\frac1v=1+\frac1w$,
\begin{align}\label{3e23}
\lf\|V_gf\r\|_{L^{q,w}(G\times\widehat{G})}
\ls\|f\|_{L^{p',u}(G)}\|g\|_{L^{p,v}(G)}.
\end{align}
Therefore, Theorem \ref{t1} is proved for any simple functions $f$ and $g$.

Assume now that $g\in L^{p,v}(G)$ and $f$ is a simple function. Then, by Proposition \ref{3p6}, it is
easy to see that there exists a Cauchy sequence of simple functions $\{g_k\}_{k\in\nn}$ such that
$$\lim_{k\to\fz}\lf\|g_k-g\r\|_{L^{p,v}(G)}=0.$$
From this and \eqref{3e23}, it follows that, as $k$, $\ell\to\fz$,
\begin{align*}
\lf\|V_{g_k}f-V_{g_\ell}f\r\|_{L^{q,w}(G\times\widehat{G})}
&=\lf\|V_{(g_k-g_\ell)}f\r\|_{L^{q,w}(G\times\widehat{G})}\\
&\ls\|f\|_{L^{p',u}(G)}\|g_k-g_\ell\|_{L^{p,v}(G)}\to0,
\end{align*}
which implies that $\{V_{g_k}f\}_{k\in\nn}$ is a Cauchy sequence in $L^{q,w}(G\times\widehat{G})$.
By this and the completeness of $L^{q,w}(G\times\widehat{G})$ (see Remark \ref{2r1}(i)), we conclude
that there exists some
$h\in L^{q,w}(G\times\widehat{G})$ such that $h=\lim_{k\to\fz}V_{g_k}f$ in $L^{q,w}(G\times\widehat{G})$.
Then let $V_{g}f:=h$. From this and \eqref{3e23}, we deduce that $V_{g}f$ is well defined and, moreover,
for any $g\in L^{p,v}(G)$,
\begin{align*}
\lf\|V_gf\r\|_{L^{q,w}(G\times\widehat{G})}
&\ls\limsup_{k\to\fz}\lf[\lf\|V_gf-V_{g_k}f\r\|_{L^{q,w}(G\times\widehat{G})}
+\lf\|V_{g_k}f\r\|_{L^{q,w}(G\times\widehat{G})}\r]\\
&=\limsup_{k\to\fz}\lf\|V_{g_k}f\r\|_{L^{q,w}(G\times\widehat{G})}\noz\\
&\ls\lim_{k\to\fz}\|f\|_{L^{p',u}(G)}\|g_k\|_{L^{p,v}(G)}=\|f\|_{L^{p',u}(G)}\|g\|_{L^{p,v}(G)}.\noz
\end{align*}
Thus, Theorem \ref{t1} is true for any $g\in L^{p,v}(G)$ and simple function $f$.

Furthermore, let $f\in L^{p',u}(G)$ and $g\in L^{p,v}(G)$. Then repeating the procedure of density extension
of $g$ as above, it is easy to check that Theorem \ref{t1} also holds true. This finishes the proof of
Theorem \ref{t1}.
\end{proof}

At the end of this section, we give the proof of Theorem \ref{t2}.

\begin{proof}[\textbf{Proof of Theorem \ref{t2}}]
By the boundedness of the STFT from $L^2(G)\times L^2(G)$ to $L^2(G)$ and from $L^2(G)\times L^2(G)$
to $L^\fz(G)$ (see Lemma \ref{3l1}), as well as Proposition \ref{3p2} and Remark \ref{2r1}(ii), we see that,
for any $(f,g)\in L^{2,1}(G)\times L^{2,1}(G)$,
$$\lf\|V_gf\r\|_{L^{2,\fz}(G\times\widehat{G})}\ls\|f\|_{L^{2,1}(G)}\|g\|_{L^{2,1}(G)}$$
and
$$\lf\|V_gf\r\|_{L^{\fz,\fz}(G\times\widehat{G})}\ls\|f\|_{L^{2,1}(G)}\|g\|_{L^{2,1}(G)}.$$
This gives us that the STFT is weak type $(2,2;2)$ and weak type $(2,2;\fz)$. By this and Proposition \ref{3p5},
we obtain that, for any simple functions $f$ and $g$,
\begin{align}\label{3e25}
\lf(V_gf\r)^*(t)\ls S_\eta\lf(f^*,g^*\r)(t),\quad\forall\,t\in(0,\fz),
\end{align}
where
\begin{align*}
\eta=\lf\{\lf(\frac12,\frac12,\frac12\r),\,\lf(\frac12,\frac12,0\r)\r\}
\end{align*}
and the corresponding Calder\'{o}n operator
\begin{align*}
S_\eta(f,g)(t)=\int_0^\fz\int_0^\fz f(r)g(s)\min
\lf\{\sqrt{\frac{rs}{t}},\sqrt{rs}\r\}\,\frac{dr}r\frac{ds}s,\quad\forall\,t\in(0,\fz).
\end{align*}
Obviously,
\begin{align*}
\min\lf\{\sqrt{\frac{rs}{t}},\sqrt{rs}\r\}=
\begin{cases}
\sqrt{rs}
&\text{when}\ t\in(0,1),\\
\sqrt{\frac{rs}{t}}
&\text{when}\ t\in[1,\fz).
\end{cases}
\end{align*}
Then, for any $q\in(2,\fz)$,
\begin{align}\label{3e26}
&\int_0^\fz t^{1/q}S_\eta\lf(f^*,g^*\r)(t)\,\frac{dt}{t}\\
&\hs\le \int_0^1t^{1/q}S_\eta\lf(f^*,g^*\r)(t)\,\frac{dt}{t}
+\int_1^\fz t^{1/q}S_\eta\lf(f^*,g^*\r)(t)\,\frac{dt}{t}.\noz\\
&\hs=\int_0^1t^{1/q}\int_0^\fz\int_0^\fz f^*(r)g^*(s)\sqrt{rs}\,\frac{dr}r\frac{ds}s\frac{dt}{t}\noz\\
&\hs\hs+\int_1^\fz t^{1/q}\int_0^\fz\int_0^\fz
f^*(r)g^*(s)\sqrt{\frac{rs}{t}}\,\frac{dr}r\frac{ds}s\frac{dt}{t}\noz\\
&\hs=:{\rm L_1}+{\rm L_2}.\noz
\end{align}
For the term ${\rm L_1}$, by \eqref{2e1}, we get
\begin{align}\label{3e27}
{\rm L_1}
&=\lf(\int_0^1t^{1/q-1}\,dt\r)
\lf[\int_0^\fz r^{1/2}f^*(r)\,\frac{dr}r\r]\lf[\int_0^\fz s^{1/2}g^*(s)\,\frac{ds}s\r]\\
&\sim\|f\|_{L^{2,1}(G)}\|g\|_{L^{2,1}(G)}.\noz
\end{align}
Similarly, for ${\rm L_2}$, we have
\begin{align*}
{\rm L_2}
&=\lf(\int_1^\fz t^{1/q-3/2}\,dt\r)
\lf[\int_0^\fz r^{1/2}f^*(r)\,\frac{dr}r\r]\lf[\int_0^\fz s^{1/2}g^*(s)\,\frac{ds}s\r]\\
&\sim\|f\|_{L^{2,1}(G)}\|g\|_{L^{2,1}(G)}.\noz
\end{align*}
Combining this, \eqref{3e26}, \eqref{3e27}, we conclude that
\begin{align*}
\int_0^\fz t^{1/q}S_\eta\lf(f^*,g^*\r)(t)\,\frac{dt}{t}
\ls\|f\|_{L^{2,1}(G)}\|g\|_{L^{2,1}(G)}.
\end{align*}
This, together with \eqref{2e1} and \eqref{3e25}, further implies that, for any $q\in(2,\fz)$
and any simple functions $f$ and $g$,
\begin{align*}
\lf\|V_gf\r\|_{L^{q,1}(G\times\widehat{G})}
&=\int_0^\fz t^{1/q}\lf(V_gf\r)^*(t)\,\frac{dt}{t}\\
&\ls\int_0^\fz t^{1/q}S_\eta\lf(f^*,g^*\r)(t)\,\frac{dt}{t}\\
&\ls\|f\|_{L^{2,1}(G)}\|g\|_{L^{2,1}(G)}.
\end{align*}

For any $(f,g)\in L^{2,1}(G)\times L^{2,1}(G)$, the valid of \eqref{2e3} can be verified by a
density argument as in the proof of Theorem \ref{t1} with some slight modifications. The proof
of Theorem \ref{t2} is completed.
\end{proof}

\begin{remark}
Taking into account the continuous inclusion
$L^{q,1}(G\times\widehat{G})\subset L^{q,w}(G\times\widehat{G})$ for any $w\in[1,\fz]$
(see Proposition \ref{3p2}), from Theorem \ref{t2}, it follows that, for any $q\in(2,\fz)$, $w\in[1,\fz]$
and $(f,g)\in L^{2,1}(G)\times L^{2,1}(G)$,
\begin{align*}
\lf\|V_gf\r\|_{L^{q,w}(G\times\widehat{G})}\ls\|f\|_{L^{2,1}(G)}\|g\|_{L^{2,1}(G)}.
\end{align*}
\end{remark}

\section{Proof of Theorem \ref{t3}}\label{s4}

We commence with displaying \cite[p.\,260, Theorem 7.7]{bs88} as a proposition.

\begin{proposition}\label{4p1'}
Let $p\in(1,\fz)$ and $u,\,v,\,w\in[1,\fz)$ such that
\begin{align*}
\frac1u+\frac1v=1+\frac1w.
\end{align*}
Then the tensor product operator $T$, which is defined by setting
$$T:\ (f,g)\mapsto f(\cdot)g(\cdot)\equiv (f\otimes g)(\cdot,\cdot),$$
satisfies
\begin{align*}
\lf\|T(f,g)\r\|_{L^{p,w}(G\times\widehat{G})}\le C\|f\|_{L^{p,u}(G)}\|g\|_{L^{p,v}(G)},
\end{align*}
where $C$ is a positive constant independent of $f$ and $g$.
\end{proposition}

The succeeding equivalent term of Lorentz quasi-norms is just \cite[Proposition 1.4.9]{gra},
which plays a key role in the proofs of both Lemmas \ref{4l2} and \ref{4l3} below.

\begin{proposition}\label{4p1}
Let $p\in(0,\fz)$ and $q\in(0,\fz]$. Then, for any $f\in L^{p,q}(G)$,
\begin{align*}
\|f\|_{L^{p,q}(G)}=
\begin{cases}
p^{1/q}\lf\{\dis\int_0^\fz\alpha^{q-1}\lf[d_f(\alpha)\r]
^{q/p}\,d\az\r\}^{1/q}
&\text{when}\ q\in(0,\fz),\\
\dis\sup_{\alpha\in(0,\fz)}
\lf\{\alpha\lf[d_f(\alpha)\r]^{1/p}\r\}
&\text{when}\ q=\fz,
\end{cases}
\end{align*}
where $d_f$ denotes the distribution function of $f$; see \eqref{2e0}.
\end{proposition}

To show Theorem \ref{t3}, we also need the following two auxiliary lemmas, which
extend \cite[Lemmas 1 and 2]{Sa14}, respectively.

\begin{lemma}\label{4l2}
Let $p\in(0,\fz)$ and $q\in(0,\fz]$. If $\tau\in\tzt$ with its modulus $\Dz_\tau\in(0,1)$, then
$$\lf\|A_\tau f\r\|_{L^{p,q}(G)}=\lf(\frac{\Dz_\tau}{1-\Dz_\tau}\r)^{1/p}\|f\|_{L^{p,q}(G)},$$
for any $f\in L^{p,q}(G)$, where the operator $A_\tau$ is defined as
\begin{align}\label{3e28}
A_\tau:\ f(x)\to f\lf(\lf(I-\tau^{-1}\r)(x)\r),\quad \forall\,x\in G.
\end{align}
\end{lemma}

\begin{proof}
Observe that, for any $\az\in[0,\fz)$,
\begin{align*}
d_{A_\tau f}(\alpha)&=\mu\lf(\lf\{x\in G:\ |A_\tau f(x)|>\alpha\r\}\r)\\
&=\mu\lf(\lf\{x\in G:\ \lf|f\lf(\lf(I-\tau^{-1}\r)(x)\r)\r|>\alpha\r\}\r)\\
&=\int_G\mathbf{1}_{\{x\in G:\ |f((I-\tau^{-1})(x))|>\alpha\}}(x)\,d\mu(x)\\
&=\frac1{|1-\Dz_{\tau}^{-1}|}\int_G\mathbf{1}_{\{x\in G:\ |f(x)|>\alpha\}}(x)\,d\mu(x)\\
&=\frac{\Dz_{\tau}}{1-\Dz_{\tau}}\mu\lf(\lf\{x\in G:\ |f(x)|>\alpha\r\}\r)
=\frac{\Dz_{\tau}}{1-\Dz_{\tau}}d_{f}(\alpha).
\end{align*}
By this and Proposition \ref{4p1}, we easily obtain the desired result. The proof of Lemma \ref{4l2}
is completed.
\end{proof}

For any $\tau\in\tzt$, its dual automorphism $\tau^*:\ \widehat{G}\to\widehat{G}$ is defined by setting
\begin{align}\label{3e29}
\lf\langle x,\tau^*(\xi)\r\rangle:=
\lf\langle \tau(x),\xi\r\rangle,\quad\forall\,x\in G,\ \xi\in\widehat{G}.
\end{align}
Moreover, there exists a constant $\Dz_{\tau^*}\in(0,\fz)$ such that, for any $\mu^*$-measurable function $f$,
$$\int_{\widehat{G}}f(\tau^*(\xi))\,d\mu^*(\xi)=\Dz_{\tau^*}^{-1}\int_{\widehat{G}}f(\xi)\,d\mu^*(\xi),$$
where $\mu^*$ is the corresponding Haar measure on $\widehat{G}$.

\begin{lemma}\label{4l3}
Let $p\in(0,\fz)$ and $q\in(0,\fz]$. If $\tau\in\tzt$ with its modulus $\Dz_\tau\in(0,1)$, then
$$\lf\|V_g^\tau f\r\|_{L^{p,q}(G)}=\lf[(1-\Dz_\tau)\Dz_{\tau^*}\r]^{1/p}\|V_g f\|_{L^{p,q}(G)},$$
for any $f\in L^{p,q}(G)$, where $V_g^\tau f$ is defined as
\begin{align}\label{3e30}
V_g^\tau f(x,\xi)
:=V_g f\lf(\lf(I-\tau\r)^{-1}(x),(\tau^{-1})^*(\xi)\r),\quad \forall\,x\in G,\ \xi\in\widehat{G}.
\end{align}
\end{lemma}

\begin{proof}
Note that, for any $\az\in[0,\fz)$,
\begin{align*}
d_{V_g^\tau f}(\alpha)
&=\mu\times\mu^*\lf(\lf\{(x,\xi)\in G\times\widehat{G}:\ \lf|V_g^\tau f(x,\xi)\r|>\alpha\r\}\r)\\
&=\mu\times\mu^*\lf(\lf\{(x,\xi)\in G\times\widehat{G}:\
\lf|V_g f\lf(\lf(I-\tau\r)^{-1}(x),(\tau^{-1})^*(\xi)\r)\r|>\alpha\r\}\r)\\
&=\mu\lf(\lf\{x\in G:\ \lf|V_g f\lf(\lf(I-\tau\r)^{-1}(x),\cdot\r)\r|>\alpha\r\}\r)\\
&\quad\times\mu^*\lf(\lf\{\xi\in \widehat{G}:\ \lf|V_g f\lf(\cdot\,,(\tau^{-1})^*(\xi)\r)\r|>\alpha\r\}\r)\\
&=\int_G\mathbf{1}_{\{x\in G:\ \lf|V_g f((I-\tau\r)^{-1}(x),\cdot)|>\alpha\}}(x)\,d\mu(x)\\
&\quad\times\int_{\widehat{G}}
\mathbf{1}_{\{\xi\in \widehat{G}:\ |V_g f(\cdot\,,(\tau^{-1})^*(\xi))|>\alpha\}}(\xi)\,d\mu^*(\xi)\\
&=(1-\Dz_\tau)\int_G\mathbf{1}_{\{x\in G:\ |V_g f(x,\cdot)|>\alpha\}}(x)\,d\mu(x)\\
&\quad\times\Dz_{\tau^*}\int_{\widehat{G}}
\mathbf{1}_{\{\xi\in \widehat{G}:\ |V_g f(\cdot,\xi)|>\alpha\}}(\xi)\,d\mu^*(\xi)\\
&=(1-\Dz_\tau)\mu\lf(\lf\{x\in G:\ \lf|V_g f(x,\cdot)\r|>\alpha\r\}\r)\\
&\quad\times\Dz_{\tau^*}
\mu^*\lf(\lf\{\xi\in \widehat{G}:\ \lf|V_g f(\cdot,\xi)\r|>\alpha\r\}\r)\\
&=\lf[(1-\Dz_\tau)\Dz_{\tau^*}\r]\mu\times\mu^*\lf(\lf\{(x,\xi)\in G\times\widehat{G}:\
\lf|V_g f\lf(x,\xi\r)\r|>\alpha\r\}\r)\\
&=\lf[(1-\Dz_\tau)\Dz_{\tau^*}\r]d_{V_gf}(\alpha).
\end{align*}
From this and Proposition \ref{4p1}, it follows the desired conclusion.
This finished proof of Lemma \ref{4l2}.
\end{proof}

Next, we prove Theorem \ref{t3}.

\begin{proof}[\textbf{Proof of Theorem \ref{t3}}]
We show the present theorem by two steps.

\textbf{Step 1.}
The aim of this step is to give the proofs of (i) and (ii). To achieve this, by \eqref{2e4},
Remark \ref{2r4}(i), \eqref{3e29} and \eqref{3e30},
we know that, for any $(f,g)\in L^{p',u}(G)\times L^{p,v}(G)$ and $(x,\xi)\in G\times\widehat{G}$,
\begin{align*}
&W_\tau(f,g)(x,\xi)\\
&\quad=\int_Gf(x+\tau(y))\overline{g(x-(I-\tau)y)}\,\overline{\langle y,\xi\rangle}\,d\mu(y)\\
&\quad=\Dz_{\tau}^{-1}\int_Gf(z)\overline{g\lf(z-\tau^{-1}(z-x)\r)}\,
\overline{\lf\langle \tau^{-1}(z-x),\xi\r\rangle}\,d\mu(z)\\
&\quad=\Dz_{\tau}^{-1}\int_Gf(z)\overline{g\lf(z-\tau^{-1}(z)+\tau^{-1}(x)\r)}\,
\overline{\lf\langle z-x,(\tau^{-1})^*(\xi)\r\rangle}\,d\mu(z)\\
&\quad=\Dz_{\tau}^{-1}\lf\langle x,(\tau^{-1})^*(\xi)\r\rangle
\int_Gf(z)\overline{g\lf(z-\tau^{-1}(z)+\tau^{-1}(x)\r)}\,
\overline{\lf\langle z,(\tau^{-1})^*(\xi)\r\rangle}\,d\mu(z)\\
&\quad=\Dz_{\tau}^{-1}\lf\langle x,(\tau^{-1})^*(\xi)\r\rangle
V_{A_\tau g} f\lf(\lf(I-\tau\r)^{-1}(x),(\tau^{-1})^*(\xi)\r)\\
&\quad=\Dz_{\tau}^{-1}\lf\langle x,(\tau^{-1})^*(\xi)\r\rangle
V_{A_\tau g}^\tau f(x,\xi),
\end{align*}
where the penultimate equality used \eqref{2e1'} and the following fact:
\begin{align*}
g\lf(z-\tau^{-1}(z)+\tau^{-1}(x)\r)
&=g\lf(z-\tau^{-1}(z)+\tau^{-1}(I-\tau)(I-\tau)^{-1}(x)\r)\\
&=g\lf(z-\tau^{-1}(z)+(\tau^{-1}-I)(I-\tau)^{-1}(x)\r)\\
&=g\lf(\lf(z-(I-\tau)^{-1}(x)\r)-\tau^{-1}\lf(z-(I-\tau)^{-1}(x)\r)\r)\\
&=A_\tau g\lf(z-(I-\tau)^{-1}(x)\r)
\end{align*}
by \eqref{3e28}. From this and Lemma \ref{4l3}, we infer that
\begin{align}\label{4e1}
\lf\|W_\tau(f,g)\r\|_{L^{q,w}(G\times\widehat{G})}
&=\lf\|\Dz_{\tau}^{-1}\lf\langle x,(\tau^{-1})^*(\xi)\r\rangle
V_{A_\tau g}^\tau f\r\|_{L^{q,w}(G\times\widehat{G})}\\
&=\Dz_{\tau}^{-1}\lf\|V_{A_\tau g}^\tau f\r\|_{L^{q,w}(G\times\widehat{G})}\noz\\
&=\Dz_{\tau}^{-1}\lf[(1-\Dz_\tau)\Dz_{\tau^*}\r]^{1/q}\lf\|V_{A_\tau g} f\r\|_{L^{q,w}(G\times\widehat{G})}.\noz
\end{align}

For (i), by \eqref{4e1}, Theorem \ref{t1'} and Lemma \ref{4l2}, we conclude that
\begin{align*}
\lf\|W_\tau(f,g)\r\|_{L^{q,w}(G\times\widehat{G})}
&\le C\Dz_{\tau}^{-1}\lf[(1-\Dz_\tau)\Dz_{\tau^*}\r]^{1/q}\|f\|_{L^{p_1,u}(G)}\lf\|A_\tau g\r\|_{L^{p_2,v}(G)}\\
&=C\Dz_{\tau}^{-1}\lf[(1-\Dz_\tau)\Dz_{\tau^*}\r]^{1/q}
\lf(\frac{\Dz_\tau}{1-\Dz_\tau}\r)^{1/p_2}\|f\|_{L^{p_1,u}(G)}\|g\|_{L^{p_2,v}(G)}\\
&\sim \|f\|_{L^{p_1,u}(G)}\|g\|_{L^{p_2,v}(G)},
\end{align*}
where $C$ is the same positive constant as in Theorem \ref{t1'}. Thus, (i) is true.
Similarly, for (ii), by Theorem \ref{t1}, we obtain
\begin{align*}
\lf\|W_\tau(f,g)\r\|_{L^{q,w}(G\times\widehat{G})}
&\le C\Dz_{\tau}^{-1}\lf[(1-\Dz_\tau)\Dz_{\tau^*}\r]^{1/q}\|f\|_{L^{p',u}(G)}\lf\|A_\tau g\r\|_{L^{p,v}(G)}\\
&=C\Dz_{\tau}^{-1}\lf[(1-\Dz_\tau)\Dz_{\tau^*}\r]^{1/q}
\lf(\frac{\Dz_\tau}{1-\Dz_\tau}\r)^{1/p}\|f\|_{L^{p',u}(G)}\|g\|_{L^{p,v}(G)}\\
&\sim \|f\|_{L^{p',u}(G)}\|g\|_{L^{p,v}(G)},
\end{align*}
where $C$ is the same positive constant as in Theorem \ref{t1}. Therefore, (ii) holds true.

\textbf{Step 2.} In this step, we prove (iii) and (iv). For (iii), let $(f,g)\in L^{p,u}(G)\times L^{p',v}(G)$.
Then, by Lemma \ref{4l1}, we have
$$\lf\|\widehat{g}\,\r\|_{L^{p,v}(G)}\ls\lf\|g\r\|_{L^{p',v}(G)}$$
since $p\in(2,\fz)$. In addition, Remark \ref{2r4}(ii) gives us
\begin{align*}
W_\theta(f,g)(x,\xi)=f(x)\overline{\langle x,\xi\rangle}\,\overline{\widehat{g}(\xi)},
\quad \forall\,(x,\xi)\in G\times\widehat{G}.
\end{align*}
Therefore, using Proposition \ref{4p1'}, we get
\begin{align*}
\lf\|W_\theta(f,g)\r\|_{L^{p,w}(G\times\widehat{G})}
&\ls\|f\|_{L^{p,u}(G)}\lf\|\widehat{g}\,\r\|_{L^{p,v}(G)}\\
&\ls\|f\|_{L^{p,u}(G)}\|g\|_{L^{p',v}(G)}.
\end{align*}
The proof of (iii) is completed.

By Remark \ref{2r4}(iii) and an argument similar to the above proof of (iii) with some slight modifications,
it is easy to see that (iv) holds true; the details are omitted. This finished proof of Theorem \ref{t3}.
\end{proof}

\section{Proof of Theorem \ref{t5}}\label{s5}

\begin{proof}[\textbf{Proof of Theorem \ref{t5}}]
We only prove (ii), the proof of (i) is similar.
Observe that $1<\frac1u+\frac1v$; see \eqref{2e6}. Then there exist some $w\in[1,\fz)$
and $r\in(1,\fz]$ such that
$$\frac1u+\frac1v\ge1+\frac1w\quad{\rm and}\quad \frac1r+\frac1w=\frac12.$$
From this, \eqref{2e7}, the assumption that $\frac1q+\frac1s=\frac12$ and the H\"{o}lder inequality on
Lorentz spaces (see Proposition \ref{5p1}), we deduce that
\begin{align}\label{5e1}
\sqrt{\varepsilon}\le
\lf\|V_gf\cdot\mathbf{1}_{\Omega}\r\|_{L^{2,2}(G\times\widehat{G})}
\le2\lf\|V_gf\r\|_{L^{q,w}(G\times\widehat{G})}\lf\|\mathbf{1}_{\Omega}\r\|_{L^{s,r}(G\times\widehat{G})}.
\end{align}
In addition, by \cite[Example 1.4.8]{gra}, we obtain
\begin{align*}
\lf\|\mathbf{1}_{\Omega}\r\|_{L^{s,r}(G\times\widehat{G})}=\left\{
\begin{array}{cl}
\dis\lf(\frac sr\r)^{1/r}\lf[\lf(\mu\times\mu^*\r)(\Omega)\r]^{1/s}
\quad &{\rm if}\ r\in(2,\fz),\\
\lf[\lf(\mu\times\mu^*\r)(\Omega)\r]^{1/s}
&{\rm if}\ r=\infty.
\end{array}\r.
\end{align*}
This, together with \eqref{5e1} and Theorem \ref{t1}, further implies that
\begin{align*}
\sqrt{\varepsilon}
&\le2C\|f\|_{L^{p',u}(G)}\|g\|_{L^{p,v}(G)}\lf\|\mathbf{1}_{\Omega}\r\|_{L^{s,r}(G\times\widehat{G})}\\
&\ls\|f\|_{L^{p',u}(G)}\|g\|_{L^{p,v}(G)}\lf[\lf(\mu\times\mu^*\r)(\Omega)\r]^{1/s},
\end{align*}
where $C$ is the same positive constant as in Theorem \ref{t1}. Thus,
$$\lf(\mu\times\mu^*\r)(\Omega)\gtrsim\varepsilon^{s/2}\|f\|_{L^{p',u}(G)}^{-1}\|g\|_{L^{p,v}(G)}^{-1}$$
with the positive control constant independent of $f,\,g$ and $\Oz$,
namely, \eqref{2e8} is true, which completes the proof of Theorem \ref{t5}.
\end{proof}

\bigskip

\noindent
\textbf{Data availability} This article has no associated data.

\smallskip

\noindent
\textbf{Conflict of interest} The authors declare that they have no Conflict of interest.

\noindent  Jun Liu and Chi Zhang

\smallskip

\noindent School of Mathematics, JCAM,
China University of Mining and Technology,
Xuzhou 221116, China

\smallskip

\noindent {\it E-mails}: \texttt{junliu@cumt.edu.cn} (J. Liu);
\texttt{zclqq32@cumt.edu.cn} (C. Zhang)

\medskip

\noindent  Yaqian Lu

\smallskip

\noindent
School of Mathematics and Statistics, Central South University,
Changsha 410075, China

\smallskip

\noindent{\it E-mail:} \texttt{yaqianlu@csu.edu.cn} (Y. Lu)

\medskip

\noindent  Xianjie Yan

\smallskip

\noindent
Institute of Contemporary Mathematics,
School of Mathematics and Statistics, Henan University,
Kaifeng 475004, China

\smallskip

\noindent{\it E-mail:} \texttt{xianjieyan@henu.edu.cn} (X. Yan)

\end{document}